\documentclass{amsart}[12pt]
\usepackage{amsthm}
\usepackage{amsfonts}
\usepackage{comment}
\usepackage{amssymb}
\usepackage{eufrak}
\usepackage{comment}
\usepackage{hyperref}
\usepackage{mathtools}
\usepackage{eucal}
\numberwithin{equation}{section}
\usepackage{breqn}
\allowdisplaybreaks
\usepackage[OT2,T1]{fontenc}
\usepackage{cleveref}
\usepackage[shortlabels]{enumitem}

\DeclareMathAlphabet{\mathdutchcal}{U}{dutchcal}{m}{n}
\SetMathAlphabet{\mathdutchcal}{bold}{U}{dutchcal}{b}{n}
\DeclareMathAlphabet{\mathdutchbcal}{U}{dutchcal}{b}{n}

\newtheorem*{theorem}{Main Theorem}
\newtheorem{Theorem}{Theorem}[section]

\newtheorem{Lemma}[Theorem]{Lemma}
\newtheorem{Proposition}[Theorem]{Proposition}

\theoremstyle{definition}
\newtheorem*{Example}{Example}
\theoremstyle{remark}
\newtheorem*{remark}{Remark}
\usepackage[margin=1.25in]{geometry}
\usepackage{cite}
\title{Module constructions for certain subgroups of the largest Mathieu group}
\author{Lea Beneish} 
\address{Lea Beneish \\
	Department of Mathematics\\
	Emory University\\
    Atlanta, GA 30322\\
    United States}
\email{leabeneish@gmail.com}
\begin{document}
\begin{abstract} For certain subgroups of $M_{24}$, we give vertex operator algebraic module constructions whose associated trace functions are meromorphic Jacobi forms. These meromorphic Jacobi forms are canonically associated to the mock modular forms of Mathieu moonshine. The construction is related to the Conway moonshine module and employs a technique introduced by Anagiannis--Cheng--Harrison. With this construction we are able to give concrete vertex algebraic realizations of certain cuspidal Hecke eigenforms of weight two. In particular, we give explicit realizations of trace functions
whose integralities are equivalent to divisibility conditions on the number of $\mathbb{F}_p$ points on the Jacobians of modular curves.

\end{abstract}
\maketitle
\section{Introduction}
Moonshine refers to unexpected connections between finite simple groups and modular forms. The first such instance was observed by McKay and Thompson in the 1970s and involves the monster group $\mathbb{M}$ and the modular $j$-invariant. This observation led to the monstrous moonshine conjecture of Thompson \cite{thompsonfmf} and Conway--Norton \cite{ConwayNorton}. This conjecture postulates the existence of an infinite-dimensional graded module \[V^\natural=\bigoplus\limits_n V^\natural_n\] such that for each $g$ in the monster group, the graded trace function \[T_g(\tau):=\sum\limits_{n=-1}^\infty \text{tr}(g\mid V_n^\natural)q^n\] is the unique modular function that generates the genus zero function field arising from a specific subgroup $\Gamma_g$ of $SL_2({\mathbb{R}})$, normalized such that $T_g(\tau)=q^{-1}+O(q)$ (where $\tau \in \mathbb{H}$ and $q=e^{2\pi i \tau}$)\cite{ConwayNorton}.

In the next few years, Frenkel, Lepowsky, and Meurman \cite{FLM83, FLM84, FLM} constructed $V^\natural$ as a vertex operator algebra (VOA). Then in 1992, Borcherds used the theory of vertex operator algebras and generalized Kac--Moody algebras (also known as Borcherds--Kac--Moody algebras \cite{BKM}) to show that $V^\natural$ has the properties conjectured by Conway and Norton and thus proved the monstrous moonshine conjecture \cite{Borcherds}. 

Since the proof of the monstrous moonshine conjecture, several other examples of moonshine phenomena have been discovered. Most significant for this work is the 2010 observation by  Eguchi, Ooguri, and Tachikawa \cite{EOT} of a connection between the largest Mathieu group $M_{24}$ and the elliptic genus of $K3$ surfaces. More precisely, they noticed that the low order multiplicities of superconformal algebra characters in the $K3$ elliptic genus are simple linear combinations of irreducible representations of $M_{24}$. This led them to conjecture that there exists an infinite-dimensional graded $M_{24}$-module  \[K^\natural=\bigoplus\limits_n K^\natural_n\] whose trace functions, denoted $H_g(\tau)$, are certain mock modular forms. We refer to \cite{hmfbook,folsom} for background on mock modular forms. In analogy with the work of Conway--Norton, work of Cheng, Eguchi--Hikami, and Gaberdiel--Hohenegger--Volpato \cite{2,3,4,5} determined the mock modular forms $H_g(\tau)$ and then in 2012 Gannon \cite{Gannon} proved the existence of the associated module $K^{\natural}$. 

The analogy between the monster and $M_{24}$ extends further when one considers the relationship between these groups and even unimodular positive-definite lattices of rank $24$. The Leech lattice $\Lambda$ \cite{Leech1, Leech2} was proven by Conway \cite{ConwayonLeech} to be the unique such lattice with no root vectors. It is closely related to the monster. In fact, the Leech lattice was involved in both the construction of the monster by Griess \cite{Griess}, and in the construction of the monster module \cite{FLM, FLM83, FLM84}. The group $M_{24}$ is closely related to another such lattice, the (unique up to isomorphism) even unimodular lattice with rank $24$ and root system $A_1^{24}$ \cite{Niemeier}; $M_{24}$ can be realized as the automorphism group of that lattice modulo the normal subgroup generated by reflections in roots. Cheng, Duncan, and Harvey \cite{umbralconjecture} conjectured that this relationship generalizes. More precisely, they formulated the umbral moonshine conjecture, stating that $M_{24}$ moonshine belongs to a class of $23$ moonshines, each corresponding to one of the $23$ Niemeier lattices with root systems of full rank \cite{Niemeier}. The existence of these umbral moonshine modules was proven in $2015$ by Duncan, Griffin, and Ono \cite{DGO}. There has been recent progress in constructing umbral moonshine modules by Anagiannis--Cheng--Harrison, Cheng--Duncan, Duncan--Harvey, and Duncan--O'Desky (see \cite{DH, DO, CDMeroJacobi, ACH}), however the umbral moonshine theory does not yet include explicit module constructions in all of its cases.

With the intention to further develop the analogy between the monster and $M_{24}$, in \cite{QM} the author associated weight $2$ quasimodular forms $Q_g(\tau)$ to the elements of $M_{24}$. These $Q_g(\tau)$ come from the holomorphic projection of $\eta^3(\tau)\hat{H}_g(\tau)$ where $\hat{H}_g(\tau)$ is the completion of the mock modular form $H_g(\tau)$. The author proves the existence of an $M_{24}$ module with the $Q_g(\tau)$ as trace functions \cite[Theorem 1]{QM}. An application of this is that the integrality of these functions is equivalent to certain divisibility conditions on the number of $\mathbb{F}_p$ points on Jacobians of modular curves \cite[Corollary 1.1]{QM}. Extending this, the author finds trace functions for modules of cyclic groups of arbitrary prime order with similar arithmetic connections \cite[Theorem 2]{QM}. For the cyclic groups of prime order, the author constructs related modules explicitly in terms of vertex operator algebras\cite[Theorem 3]{QM}. However, that construction came at the expense of the divisibility conditions. A major objective of this work is to amend that and give a vertex operator algebra construction for modules whose trace functions exhibit divisibility conditions.

In this work, we offer a modification of the functions $Q_g(\tau)$ and give a module construction that retains the arithmetic information mentioned above. Namely, we add $\chi(g)\left(\eta^{3}(\tau)\mu(\tau,z)+2F_2(\tau)\right)$ to $Q_g(\tau)$, where $\eta(\tau)$ is the Dedekind eta function, $\mu(\tau,z)$ is an Appell--Lerch sum, $\chi(g)$ is the number of fixed points of $g$ in the $24$-dimensional
permutation representation of $M_{24}$, and $F^{}_2(\tau)$ is defined to be \[F_2(\tau):=\sum\limits_{\substack{r>s>0\\ r-s \text{ odd}}}sq^{rs/2}.\]
This allows us to define the following meromorphic Jacobi forms \[{M}_g(\tau,z):=Q_g(\tau)+\chi(g)\left(\eta^{3}(\tau)\mu(\tau,z)+2F_2(\tau)\right),\]
associated to each element $g\in M_{24}$. In Proposition \ref{exists}, we prove the existence of a module for which suitable expansions of the $M_g(\tau,z)$ are trace functions. Because of their relation to $Q_g(\tau)$, these trace functions contain arithmetic information.

Although on the surface, the relationship between $H_g(\tau)$ and $M_g(\tau,z)$ may seem distant, we claim it is natural. We show that one can equivalently define $M_g(\tau,z)$ by writing $M_g(\tau,z)=H_g(\tau)\eta^{3}(\tau)+\chi(g)\eta^{3}(\tau)\mu(\tau,z)$. This type of expression is an example of a canonical decomposition of a meromorphic Jacboi form into a ``finite'' part and ``polar'' part established by Zwegers in \cite{Zweg} and Dabholkar, Murthy, and Zagier in \cite{DMZ}. The relationship between meromorphic Jacobi forms and umbral moonshine was first discussed by Cheng, Duncan, and Harvey in \cite{CDHLattices}. A special case of this is that the mock modular forms $H_g(\tau)$ occur as the ``finite parts'' of meromorphic Jacobi forms. 

 Finding a construction of a module whose trace functions are meromorphic Jacobi forms associated to vector valued mock modular forms in umbral moonshine is considered a natural alternative to finding a construction of a module whose trace functions are the vector valued mock modular forms. In fact, to describe this, Duncan and O'Desky coined the term ``meromorphic module problem'' in \cite{DO} when they solved this problem for the cases of umbral moonshine corresponding to the Niemeier lattices with root systems $A_6^{\oplus 4}$ and $A_{12}^{\oplus 2}$ (and partially for the cases corresponding to $A_3^{\oplus 8}$ and $A_4^{\oplus 6}$). 

For $g\in M_{24}$ such that $[g]\neq$ $3$B, $4$C, $6$B, $12$B, $21$A, $21$B, $23$A, or $23$B, (where we use the ATLAS notation in \cite{conway1985atlas} for conjugacy classes of $M_{24}$), we give concrete constructions of modules whose trace functions are $\widetilde{M}_g(\tau,z)$, where the $\widetilde{M}_g(\tau,z)$ are defined to be the Fourier expansions of  $M_g(\tau,z)$
 in the domain $0<-\text{Im}(z)<\text{Im}(\tau)$. 
\begin{theorem} For subgroups of $M_{24}$ consisting only of elements $g\in M_{24}$ such that $[g]\neq$ 3B, 6B, 12B, 21A, 21B, 23A, or 23B, and such that each element fixes a 4-dimensional space in the $24$-dimensional permutation representation of $M_{24}$, we have the following module construction:
\[\widetilde{A}(\mathfrak{p})^{}_{\emph{tw}}\otimes \mathdutchcal{W}(\mathfrak{b})_{\emph{tw}} \otimes T\] is an infinite-dimensional, bigraded, virtual module with trace functions as follows:
\begin{equation*}
\lim\limits_{\gamma \to -1}\emph{tr}\left(\widehat{g}\mathfrak{z}\widetilde{p}(0)\gamma^{J_{12}(0)}y^{J(0)}q^{L(0)-c/24}\mid \widetilde{A}(\mathfrak{p})^{}_{\emph{tw}}\otimes \mathdutchcal{W}(\mathfrak{b})_{\emph{tw}} \otimes T\right)=\widetilde{M_g}(\tau,z).\end{equation*}

\end{theorem}

\noindent We describe this construction in two steps. First we construct a related module, $\widetilde{A}(\mathfrak{p})^{}_{\text{tw}}\otimes \mathdutchcal{W}(\mathfrak{b})_{\text{tw}} \otimes V^{s\natural}_{\text{tw}}$, which is the tensor product of a Clifford module, a Weyl module, and a Conway module of \cite{DMCconway}. For their definitions and the definitions of the operators that we take the trace of, see Section \ref{mod1}. This  gives module constructions for subgroups of $M_{24}$ that do not contain elements in the conjugacy classes 3B, 4C, 6B, 12B, 21A, 21B, 23A, or 23B and such that the $24$-dimensional permutation representation of $M_{24}$ has a fixed $4$-dimensional space when restricted to that subgroup. For example, this gives a module construction for group $\bf{L_3(4)}\simeq \bf{M_{21}}$, one of the simple subgroups of $M_{24}$. This does not, however, give a module construction for $\bf{M_{11}}$ because the $24$-dimensional permutation representation of $M_{24}$ restricted to $\bf{M_{11}}$ only fixes a $3$-dimensional space.

To remedy cases such as $\bf{M_{11}}$ and arrive at the module constructions given in the Main Theorem, we apply a method of Anagiannis, Cheng, and Harrison \cite{ACH}. For these subgroups, we still require that each element of the subgroup fixes a $4$-space but not that the whole subgroup fixes the same $4$-space (we still require that the subgroup does not contain elements in the aforementioned conjugacy classes). Here we use $\widetilde{A}(\mathfrak{p})^{}_{\text{tw}}$ and $\mathdutchcal{W}(\mathfrak{b})_{\text{tw}}$ as before (defined in Section \ref{mod1}) and $T$ is a modification of $V^{s\natural}_{\text{tw}}$ which we define in Section \ref{mod2}. This, for example, gives module constructions for $\bf{M_{22}\colon 2}$ a maximal subgroup of $M_{24}$, for the smallest sporadic simple group $\bf{M_{11}}$, and for groups $\bf{2^4\colon A_7}$ and $\bf{A_8}$ which are maximal subgroups of $M_{23}$.

 The module construction for $\bf{M_{11}}$ gives an explicit realization of the trace functions $\widetilde{M}_g(\tau,z)$ whose integrality is equivalent to divisibility conditions on the number of $\mathbb{F}_p$ points on the Jacobian of the modular curve $X_0(11)$, denoted $J_0(11)$. The same is true with  $\bf{M_{22}\colon 2}$ and $\bf{2^4\colon A_7}$ for $J_0(14)$ and with $\bf{A_8}$ for $J_0(15)$.

Note that our module gives an explicit construction of the restriction of the Mathieu moonshine module to the subgroup $\bf{2^4\colon A_7}$, which has also played a prominent role in the symmetry surfing program initiated by Taormina and Wendland in \cite{TW10, TW13,TW15a,TW15b,TaoWend}. It would be interesting to compare our method to theirs.

This paper is organized as follows. In Section \ref{thefunctions}, we describe the meromorphic Jacobi forms and show that there exists a virtual $M_{24}$ module for which suitable expansions of these meromorphic Jacobi forms are the trace functions. In Section \ref{mod1}, we construct modules for subgroups of $M_{24}$ that exclude certain conjugacy classes and such that the $24$-dimensional permutation representation of $M_{24}$ restricted to that subgroup has a fixed $4$-dimensional space. In Section \ref{mod2}, we prove the Main Theorem by applying a method of Anagiannis, Cheng, and Harrison \cite{ACH} to modify the construction in Section \ref{mod1} so that we can replace the condition that the subgroup must fix a $4$-space with the condition that only each element in that subgroup fixes a $4$-space.
\vspace{-1mm}
\subsection*{Acknowledgements} The author is grateful to John Duncan for suggesting the topic and for his valuable advice and comments on earlier drafts. The author would also like to thank Jackson Morrow for comments on an earlier draft.

\section{The functions}\label{thefunctions}
In this section we describe the meromorphic Jacobi forms $M_g(\tau,z)$. We will explicitly construct modules for which suitable expansions of the $M_g(\tau,z)$ are the trace functions. The module constructions can be found in Sections \ref{mod1} and \ref{mod2}, but first we prove the existence of an overarching virtual $M_{24}$-module.

In order to define the meromorphic Jacobi forms, we recall a few definitions. Let $\eta(\tau)$ be the Dedekind eta function, defined by
\begin{equation} \eta(\tau):=q^{1/24}\prod\limits_{n>0}(1-q^n). \end{equation}
We have the usual Jacobi theta function $\theta_1(\tau,z)$, defined as
\begin{equation} \theta_1(\tau,z):=-i q^{\frac{1}{8}}y^{\frac{1}{2}}\prod\limits_{n>0} (1-y^{-1}q^{n-1})(1-yq^n)(1-q^n), \end{equation}
where $q=e^{2\pi i \tau}$ and $y=e^{2\pi i z}$.
The Appell-Lerch sum $\mu(\tau,z)$ is given by 
\begin{equation}\mu(\tau,z):= \frac{-i y^{1/2}}{\theta_1(\tau,z)}\sum\limits_{n\in \mathbb{Z}} (-1)^n \frac{y^nq^{n(n+1)/2}}{1-yq^n}.\end{equation}

We recall that $\chi(g)$ is the number of fixed points of $g$ in the $24$-dimensional
permutation representation of $M_{24}$, the mock modular forms of weight $1/2$ associated to $g\in M_{24}$ from Mathieu moonshine are denoted by $H_g(\tau)$, and $F^{}_2(\tau)$ is defined as follows \begin{equation}F_2(\tau):=\sum\limits_{\substack{r>s>0\\ r-s \text{ odd}}}sq^{rs/2}.\end{equation}

The quasimodular forms $Q_g(\tau)$, for $g \in M_{24}$, that are the holomorphic projection of the completion of the $H_g(\tau)$ multiplied by $\eta^3(\tau)$ (i.e. $\pi_{\text{hol}}(\widehat{H}_g(\tau)\eta^3(\tau))$ from \cite{QM}) can be defined as \begin{equation}\label{Qg} Q_g(\tau):=H_{g}(\tau)\eta^3(\tau)-2\chi(g)F^{}_2(\tau).\end{equation}

We now define \begin{equation}\phi_{-2,1}(\tau,z):=-\dfrac{\theta_1^2(\tau,z)}{\eta^6(\tau)},\end{equation} and \begin{equation}\phi_{0,1}(\tau,z):=\frac{1}{2} Z_{K3}(\tau,z),\end{equation} where, from \cite{EOT}, we have \begin{equation} Z_{K3}(\tau,z):=24\mu(\tau,z) \dfrac{\theta_1^2(\tau,z)}{\eta^3(\tau)}+H_e(\tau)\dfrac{\theta_1^2(\tau,z)}{\eta^3(\tau)}. \end{equation}

\noindent More generally, for $g\in M_{24}$, we define the following weak Jacobi forms:  \begin{equation} Z_g(\tau,z):=\chi(g)\mu(\tau,z)\dfrac{\theta_1^2(\tau,z)}{\eta^3(\tau)}+H_g(\tau)\dfrac{\theta_1^2(\tau,z)}{\eta^3(\tau)}.\end{equation}

In \cite{DMCderived}, Duncan and Mack-Crane associate weak Jacobi forms $\phi_g(\tau,z)$ of weight zero and index one to  symplectic derived equivalences of projective complex K3 surfaces that fix a stability condition in the distinguished space identified by Bridgeland. They identify such automorphisms with elements of Aut($\Lambda$) (the Conway group $Co_0$) fixing a sublattice of rank greater than or equal to $4$. Since $M_{24}$ is a subgroup of $Co_0$, it is natural to compare the $\phi_g(\tau,z)$ to the weak Jacobi forms $Z_g(\tau,z)$ associated to $g\in M_{24}$, and in fact, these $\phi_g(\tau,z)$ are equal to $Z_g(\tau,z)$ for $g$ in all but $7$ of the $26$ conjugacy classes of $M_{24}$ (those conjugacy classes are: 3B, 4C, 6B, 12B, 21A, 21B, 23A and 23B). For an explicit expression of $\phi_g(\tau,z)$, see equation \eqref{phig}.
 
These $\phi_g(\tau,z)$ are closely related to weight two modular forms $F_g(\tau)$ (not to be confused with $F_2(\tau)$, which is not modular). We take the expression given in Proposition $9.3$ of \cite{DMCderived} as the definition of $F_g(\tau)$:
\begin{equation}{F_g}(\tau)= \dfrac{\phi_g(\tau,z)}{\phi_{-2,1}(\tau,z)}-\dfrac{1}{12}\chi(g)\dfrac{\phi_{0,1}(\tau,z)}{\phi_{-2,1}(\tau,z)}. \end{equation}
We refer to equation $(9.19)$ of \cite{DMCderived} for another definition of $F_g(\tau)$.

We begin with the following proposition in which we combine a result of Dabholkar, Murthy, and Zagier \cite{DMZ}, a result of Duncan and Mack-Crane \cite{DMCderived}, and the functions $Q_g(\tau)$ that were defined by the author in \cite{QM}.

\begin{Proposition}\label{prop21}For $g\in M_{24}$ such that $[g] \neq$ 3B, 4C, 6B, 12B, 21A, 21B, 23A and 23B we have \[\dfrac{\eta^6(\tau)\phi_g(\tau,z)}{\theta_1^2(\tau,z)}=\chi(g)\left(\eta^{3}(\tau)\mu(\tau,z)+2F_2(\tau)\right)+Q_g(\tau)\]
\end{Proposition}
\begin{proof}
From equation $(8.52)$ in \cite{DMZ} we have following: 

\begin{equation}\eta^{-3}(\tau) \dfrac{\phi_{0,1}(\tau,z)}{\phi_{-2,1}(\tau,z)}=-\dfrac{12}{\theta_1(\tau,2z)} Av^{(2)}\left[\frac{1+y}{1-y}\right]-h^{(2)}(\tau). \end{equation}

We note that $2h^{(2)}(\tau)=H_e(\tau)$ and $Av^{(2)}\left[\frac{1+y}{1-y}\right]=\theta_1(\tau,2z)\mu(\tau,z),$ (see Example $2$, Section $8.5$ of \cite{DMZ}) and so we equivalently have

\begin{equation} \eta^{-3}(\tau) \dfrac{\phi_{0,1}(\tau,z)}{\phi_{-2,1}(\tau,z)}=-\dfrac{12}{\theta_1(\tau,2z)} \theta_1(\tau,2z)\mu(\tau,z)-\frac{1}{2}H_e(\tau), \end{equation}

which can be rearranged as follows:

\begin{equation}\label{dmzre} \dfrac{\phi_{0,1}(\tau,z)}{\phi_{-2,1}(\tau,z)}=-12\eta^{3}(\tau)\mu(\tau,z)-\frac{1}{2}\eta^{3}(\tau)H_e(\tau).\end{equation}

Now, from Proposition $9.3$ of \cite{DMCderived}, we have that %for $X$ a projective complex $K3$ surface and $\sigma \in \text{Stab}^{\mathrm{o}}(X)$, that for $g\in \text{Aut}_s(D^b(X), \sigma)$, we have

\begin{equation} \phi_g(\tau,z)=\dfrac{1}{12}\chi(g)\phi_{0,1}(\tau,z)+F_g(\tau)\phi_{-2,1}(\tau,z), \end{equation}

which is equivalent to 
\begin{equation} \dfrac{\phi_g(\tau,z)}{\phi_{-2,1}(\tau,z)}=\dfrac{1}{12}\chi(g)\dfrac{\phi_{0,1}(\tau,z)}{\phi_{-2,1}(\tau,z)}+F_g(\tau). \end{equation}

Substituting the right hand side of equation \eqref{dmzre} for $\dfrac{\phi_{0,1}(\tau,z)}{\phi_{-2,1}(\tau,z)}$ we obtain

\begin{equation} \dfrac{\phi_g(\tau,z)}{\phi_{-2,1}(\tau,z)}=\dfrac{1}{12}\chi(g)\left(-12\eta^{3}(\tau)\mu(\tau,z)-\frac{1}{2}\eta^{3}(\tau)H_e(\tau)\right)+F_g(\tau). \end{equation}

Then we use the identity $\phi_{-2,1}(\tau,z)=-\dfrac{\theta_1^2(\tau,z)}{\eta^6(\tau)}$ to rewrite the above as follows:

\begin{equation} \dfrac{-\eta^6(\tau)\phi_g(\tau,z)}{\theta_1^2(\tau,z)}=\dfrac{1}{12}\chi(g)\left(-12\eta^{3}(\tau)\mu(\tau,z)-\frac{1}{2}\eta^{3}(\tau)H_e(\tau)\right)+F_g(\tau), \end{equation}

and simplifying further, we find
\begin{equation} \dfrac{-\eta^6(\tau)\phi_g(\tau,z)}{\theta_1^2(\tau,z)}=-\chi(g)\eta^{3}(\tau)\mu(\tau,z)-\frac{\chi(g)}{24}\eta^{3}(\tau)H_e(\tau)+F_g(\tau). \end{equation}

We recall the formula (from Appendix B of \cite{DGO}, also in \cite{2,3,4,5,holographic}) relating $H_g(\tau)$ and $H_e(\tau)$, for $g\in M_{24}$:

\begin{equation}\label{HeHg} H_g(\tau)\eta^3(\tau)=\dfrac{\chi(g)}{24}H_e(\tau)\eta^3(\tau)-F_g(\tau).\end{equation}

This formula gives us that 
\begin{equation} \dfrac{-\eta^6(\tau)\phi_g(\tau,z)}{\theta_1^2(\tau,z)}=-\chi(g)\eta^{3}(\tau)\mu(\tau,z)-H_g(\tau)\eta^{3}(\tau), \end{equation}
or equivalently,
\begin{equation} \dfrac{\eta^6(\tau)\phi_g(\tau,z)}{\theta_1^2(\tau,z)}=\chi(g)\eta^{3}(\tau)\mu(\tau,z)+H_g(\tau)\eta^{3}(\tau). \end{equation}

Combining this with the quasimodular forms associated to $M_{24}$ in equation \eqref{Qg}, we can write

 \begin{equation} \dfrac{\eta^6(\tau)\phi_g(\tau,z)}{\theta_1^2(\tau,z)}=\chi(g)(\eta^{3}(\tau)\mu(\tau,z)+2F_2(\tau))+Q_g(\tau).\end{equation}\end{proof}
  
 \begin{Lemma} For $g\in M_{24}$ such that $[g] \neq$ 3B, 4C, 6B, 12B, 21A, 21B, 23A and 23B, the $F_g(\tau)$ in equation \eqref{HeHg} (and Appendix B of \cite{DGO}) are the same as the $F_g(\tau)$ in equation $(9.19)$ of \cite{DMCderived}.
 \end{Lemma}
 
 \begin{proof}
 For clarity, in the proof of this lemma exclusively, we will write $\widetilde{F_g}(\tau)$ when referring to the $F_g(\tau)$ in \cite{DMCderived} and we will write $F_g(\tau)$ for the $F_g(\tau)$ in \cite{DGO}. We will show that $\widetilde{F_g}(\tau)=F_g(\tau)$.
 
 From Eguchi, Ooguri, and Tachikawa we have the following expression for the $K3$ elliptic genus
 \begin{equation} Z_{K3}(\tau,z)=24\mu(\tau,z) \dfrac{\theta_1^2(\tau,z)}{\eta^3(\tau)}+H_e(\tau)\dfrac{\theta_1^2(\tau,z)}{\eta^3(\tau)}. \end{equation}
 Rearranging the terms of the above equation, we find that 
  \begin{equation} \dfrac{Z_{K3}(\tau,z)\eta^3(\tau)}{\theta_1^2(\tau,z)}=24\mu(\tau,z) +H_e(\tau), \end{equation}
and then multiplying by $\eta^3(\tau)$ and solving for $H_e(\tau)\eta^3(\tau)$ we have 
  %\begin{equation} \dfrac{Z_{K3}(\tau,z)\eta^6(\tau)}{\theta_1^2(\tau,z)}=24\mu(\tau,z)\eta^3(\tau) +H_e(\tau)\eta^3(\tau). \end{equation}
  
 %We can then solve for $H_e(\tau)\eta^3(\tau)$ as follows:
   \begin{equation} H_e(\tau)\eta^3(\tau)=\dfrac{Z_{K3}(\tau,z)\eta^6(\tau)}{\theta_1^2(\tau,z)}-24\mu(\tau,z)\eta^3(\tau). \end{equation}
  We can then substitute the right side of the equation above for $H_e(\tau)\eta^3(\tau)$ in equation \eqref{HeHg} and we obtain:
  
 \begin{equation}
 H_g(\tau)\eta^3(\tau)=\dfrac{\chi(g)}{24}\left(\dfrac{Z_{K3}(\tau,z)\eta^6(\tau)}{\theta_1^2(\tau,z)}-24\mu(\tau,z)\eta^3(\tau)\right)-F_g(\tau).
 \end{equation}
Finally, we solve for $F_g(\tau)$ as follows:
 \begin{equation} F_g(\tau)= \dfrac{\chi(g)}{24}\dfrac{Z_{K3}(\tau,z)\eta^6(\tau)}{\theta_1^2(\tau,z)}-\chi(g)\mu(\tau,z)\eta^3(\tau)-H_g(\tau)\eta^3(\tau).
 \end{equation}
On the other hand, we recall the expression in Proposition $9.3$ of \cite{DMCderived}:

\begin{equation}\widetilde{F_g}(\tau)= \dfrac{\phi_g(\tau,z)}{\phi_{-2,1}(\tau,z)}-\dfrac{1}{12}\chi(g)\dfrac{\phi_{0,1}(\tau,z)}{\phi_{-2,1}(\tau,z)}. \end{equation}
We make the substitutions $\phi_{-2,1}(\tau,z)=-\dfrac{\theta_1^2(\tau,z)}{\eta^6(\tau)}$ and $\phi_{0,1}(\tau,z)=\dfrac{1}{2}Z_{K3}(\tau,z)$ and arrive at:
\begin{equation}\widetilde{F_g}(\tau)= \dfrac{-\phi_g(\tau,z)\eta^6(\tau)}{\theta_1^2(\tau,z)}+\dfrac{\chi(g)}{24}\dfrac{Z_{K3}(\tau,z)\eta^6(\tau)}{\theta_1^2(\tau,z)}. \end{equation}
For $g \in M_{24}$ such that $Z_g(\tau,z)=\phi_g(\tau,z)$, we can substitute $\phi_g(\tau,z)$ for the right hand side of the equation below: \begin{equation} Z_g(\tau,z)=\chi(g)\mu(\tau,z)\dfrac{\theta_1^2(\tau,z)}{\eta^3(\tau)}+H_g(\tau)\dfrac{\theta_1^2(\tau,z)}{\eta^3(\tau)}.\end{equation}
Thus we have 
\begin{equation}
\widetilde{F_g}(\tau)=-\dfrac{\eta^6(\tau)}{\theta_1^2(\tau,z)}\left(\chi(g)\mu(\tau,z)\dfrac{\theta_1^2(\tau,z)}{\eta^3(\tau)}+H_g(\tau)\dfrac{\theta_1^2(\tau,z)}{\eta^3(\tau)}\right) +\dfrac{\chi(g)}{24}\dfrac{Z_{K3}(\tau,z)\eta^6(\tau)}{\theta_1^2(\tau,z)},
\end{equation}
which simplifies to
\begin{equation} \widetilde{F_g}(\tau)=\dfrac{\chi(g)}{24}\dfrac{Z_{K3}(\tau,z)\eta^6(\tau)}{\theta_1^2(\tau,z)}-\chi(g)\mu(\tau,z)\eta^3(\tau)-H_g(\tau)\eta^3(\tau).
\end{equation}
Therefore, we see that $ \widetilde{F_g}(\tau)={F_g}(\tau)$. \end{proof}
 
Now we have described everything we need to define the functions $M_g(\tau,z)$ as follows:
 \begin{equation} M_g(\tau,z):=H_g(\tau)\eta^{3}(\tau)+\chi(g)\eta^{3}(\tau)\mu(\tau,z).\end{equation} We next show that there exists an $M_{24}$-module for which suitable expansions of the $M_g(\tau,z)$ are the graded trace functions. We define $\widetilde{M}_g(\tau,z)$ to be the expansion of $M_g(\tau,z)$ in the domain $0<-\text{Im}(z)<\text{Im}(\tau)$ ($\tau \in \mathbb{H}, z\in \mathbb{C}$).
 
 \begin{Proposition}\label{exists} There exists a virtual bigraded $M_{24}$-module \[V= \bigoplus\limits_{\substack{n,r\in \mathbb{Z}\\ n\geq 0}} V_{n,r}\] such that
\[ \widetilde{M}_g(\tau,z)=  \sum_{n,r} \emph{tr}(g\mid V_{n,r})y^rq^n.\]
%where $\widetilde{M}_g(\tau,z)$ is defined to be the expansion of  \[M_g(\tau,z):=H_g(\tau)\eta^{3}(\tau)+\chi(g)\eta^{3}(\tau)\mu(\tau,z)\]
 %in the domain $0<-\text{Im}(z)<\text{Im}(\tau)$. 
 \end{Proposition}
 \begin{proof}
 For this proof, we restrict to the domain $0<-\text{Im}(z)<\text{Im}(\tau)$. First we show that the $\widetilde{M_g}(\tau,z)$ have integral coefficients. Gannon \cite{Gannon} shows that the functions $H_{g}(\tau)$ have integral coefficients (in all of $\mathbb{H}$, and thus in the domain we specify). It follows from the definition of $\eta(\tau)$ that $\eta^3(\tau)$ has integral coefficients. It remains to show that $\chi(g)\eta^{3}(\tau)\mu(\tau,z)$ has integral coefficients (and from the definition we know the $\chi(g)$ are integers). Because the specified expansion of $\mu(\tau,z)$ is one of the $N=4$ characters \cite{EOT}, its expansion is known to have integral coefficients. Thus we conclude that the $\widetilde{M_g}(\tau,z)$ have integral coefficients.
  
  %Showing that $\chi(g)\eta^{3}(\tau)\mu(\tau,z)$ expanded in our domain has integral coefficients requires slightly more work. For this, we will find an alternative expression for the $M_g(\tau,z)$.
  
% To do this, we first recall the definition of $Z_g(\tau,z)$, the weak Jacobi forms of Mathieu moonshine \begin{equation} Z_g(\tau,z)=\chi(g)\mu(\tau,z)\dfrac{\theta_1^2(\tau,z)}{\eta^3(\tau)}+H_g(\tau)\dfrac{\theta_1^2(\tau,z)}{\eta^3(\tau)}.\end{equation}
 
 %We can re-arrange the above equation as follows
 %\begin{equation}\label{rearranged} \dfrac{Z_g(\tau,z)\eta^6(\tau)}{\theta_1^2(\tau,z)}=\chi(g)\mu(\tau,z)\eta^3(\tau)+H_g(\tau)\eta^3(\tau).\end{equation}

%Note that the right hand side of the above equation is precisely our definition of the $M_g(\tau,z)$ so that equation \eqref{rearranged} gives us another way to express $M_g(\tau,z)$.

%Now, in our domain the Fourier expansions of $Z_g(\tau,z)$, for $g\in M_{24}$ are known to have integral coefficients. We also know $\eta^6(\tau)$ has integral coefficients and in the domain to which we have restricted the expansion of $\dfrac{1}{\theta_1^2(\tau,z)}$ has integral coefficients. 
%Thus we have that the left hand side and right most term of \eqref{rearranged} have integral coefficients. So we can conclude that the term $\chi(g)\mu(\tau,z)\eta^3(\tau)$ must also have integral coefficients, and 
%Thus we conclude that the $\widetilde{M_g}(\tau,z)$ have integral coefficients.

Next we show that the multiplicities $\text{m}^M_i(n)$ of the $M_{24}$ irreducible representations in the class functions defined by the coefficients of $\widetilde{M_g}(\tau,z)$ are integral. 

Gannon shows that the multiplicity generating function \begin{equation}\label{gannonmi}\sum\limits_{n>0} \text{m}^H_i(n) q^n=\dfrac{1}{|M_{24}|}\sum\limits_{g\in M_{24}} H_g(\tau)\overline{\chi_i(g)}\end{equation} (with $\chi_i$ an irreducible character of $M_{24}$) has integral coefficients.
We need to show that the coefficients $\text{m}^M_i(n)$ are integral, where
\begin{equation}\label{mi}\sum\limits_{n>0} \text{m}^M_i(n) q^n=\dfrac{1}{|M_{24}|}\sum\limits_{g\in M_{24}} \left[H_g(\tau)\eta^3(\tau)-\chi(g)\mu(\tau,z)\eta^3(\tau)\right]\overline{\chi_i(g)}.\end{equation}
To do this, we can split the right hand side of equation (\ref{mi}) into two parts. First consider $\dfrac{1}{|M_{24}|}\sum\limits_{g\in M_{24}} H_g(\tau)\eta^3(\tau)\overline{\chi_i(g)}$. This differs from (\ref{gannonmi}) only from multiplying by $\eta^3(\tau)$, which does not change the integrality. So it suffices to show that $\dfrac{1}{|M_{24}|}\sum\limits_{g\in M_{24}}\chi(g)\mu(\tau,z)\eta^3(\tau)\overline{\chi_i(g)}$ has integral coefficients.

This is the same as showing that \[\mu(\tau,z)\eta^3(\tau)\dfrac{1}{|M_{24}|}\sum\limits_{g\in M_{24}} \chi(g)\overline{\chi_i(g)}=\mu(\tau,z)\eta^3(\tau)\langle \chi, \chi_i \rangle\] has integral coefficients. We already know that $\mu(\tau,z)\eta^3(\tau)$ has integral coefficients (see above). The integrality of $\langle \chi, \chi_i \rangle$ can be seen from the fact that $\chi(g)$ is a character of a module, and so $\langle \chi, \chi_i \rangle$ is the multiplicity of $\chi_i$ in $\chi$, which is necessarily integral. Thus the $\text{m}^M_i(n)$ from (\ref{mi}) are integral. \end{proof}
\begin{remark} In what follows, we will give module constructions such that the graded trace functions on those modules are equal to $\widetilde{M_g}(\tau,z)$. The condition that $\tau \in \mathbb{H}$, $z \in \mathbb{C}$ be such that $0<-\text{Im}(z)<\text{Im}(\tau)$ is necessary to ensure convergence of the graded dimension functions of the modules. In particular, this is what will allow us to identify the series expansions as their graded dimension functions. We will adopt this restriction of the domain for the rest of the paper.
\end{remark}
 \section{Module construction I}\label{mod1}
 In this section, for $g\in M_{24}$ such that $[g]\neq$ 3B, 4C, 6B, 12B, 21A, 21B, 23A, or 23B, we explicitly construct a module whose trace functions are the $\widetilde{M_g}(\tau,z)$ (see Proposition \ref{exists}). This will lead to module constructions for certain subgroups of $M_{24}$ with no elements in any of the above conjugacy classes. We also require that the $24$-dimensional permutation representation of $M_{24}$ has a fixed $4$-dimensional subspace when restricted to that subgroup.
 
 We use the fact that when $g\in M_{24}$, the following holds: 
 %\begin{equation} M_g(\tau,z)=\dfrac{Z_g(\tau,z)\eta^6(\tau)}{\theta_1^2(\tau,z)},
 %\end{equation}
 %and when $g\in M_{24}$ is not in one of the conjugacy classes above, we have the equality below:
 By definition of $Q_g(\tau)$, we see that ${M}_g(\tau,z)=Q_g(\tau)+\chi(g)\left(\eta^{3}(\tau)\mu(\tau,z)+2F_2(\tau)\right)$, and by Proposition \ref{prop21} of Section \ref{thefunctions}, for $g$ not in the excluded conjugacy classes as above, we have
 \begin{equation}M_g(\tau,z)=\dfrac{\phi_g(\tau,z)\eta^6(\tau)}{\theta_1^2(\tau,z)}.
 \end{equation}
 We will split this equation into three factors as follows:
  \begin{equation}\label{factors}M_g(\tau,z)=\left(\phi_g(\tau,z)\right)\left(\eta^4(\tau)\right)\left(\dfrac{\eta^2(\tau)}{\theta_1^2(\tau,z)}\right).
 \end{equation}
 
 We postpone the discussion about how to recover the first factor of $M_g(\tau,z)$ for now. The next two lemmas indicate how to recover the second and third of the three factors in Equation \ref{factors}.
 
In order to recover the second factor in \eqref{factors}, we need a module with graded dimension function $\eta^{4}(\tau)$. This can be achieved using a Clifford module vertex superalgebra. For this construction we follow Duncan and Harvey \cite{DH}. We note that the description below can also be found in \cite{QM}. In this setting, let $\mathfrak{p}$ be a one dimensional complex vector space with a non-degenerate symmetric bilinear form. Let ${\hat{\mathfrak{p}}}_{}:=\mathfrak{p}[t,t^{-1}]t^{\frac{1}{2}}$ and $\hat{\mathfrak{p}}_{\text{tw}}:=\mathfrak{p}[t,t^{-1}]$, for $a \in \mathfrak{p}$ we write $a(r)$ for $at^r$ with the bilinear form extended so that $\langle a(r), b(s) \rangle=\langle a,b \rangle \delta_{r+s,0}$.

 We define $\text{Cliff}(\hat{\mathfrak{p}})$ to be the Clifford algebra attached to $\hat{\mathfrak{p}}$. Let $\hat{\mathfrak{p}}^+:=\mathfrak{p}[t]t^{\frac{1}{2}}$ and let $\langle \hat{\mathfrak{p}}^+ \rangle$ be the subalgebra of the Clifford algebra $\text{Cliff}(\hat{\mathfrak{p}})$ generated by $\hat{\mathfrak{p}}^+$. Take $\mathbb{C}\textbf{v}$ to be a $\langle \hat{\mathfrak{p}} ^+\rangle$ module such that $1\textbf{v}=\textbf{v}$ and $p(r)\textbf{v}=0$ for $r>0$.
 Then we define
 \begin{equation}A(\mathfrak{p})_{}^{}:=\text{Cliff}(\hat{\mathfrak{p}} )\otimes_{\langle \hat{\mathfrak{p}}^{+}\rangle} \mathbb{C}\textbf{v},\end{equation}
 and $A(\mathfrak{p})$ has the structure of a vertex superalgebra such that $Y(u\left(-\frac{1}{2}\right)\textbf{v},z)=\sum\limits_{n\in \mathbb{Z}} u(n+\frac{1}{2})z^{-n-1}$ for $u\in \mathfrak{p}$. $A(\mathfrak{p})$ has the structure of a vertex operator superalgebra with central charge $\frac{1}{2}$ when we equip it with the Virasoro element \[ \omega:=%\dfrac{-1}{4}\sum\limits_{i=1}^{24\ell_N}
p\left(\frac{-3}{2}\right)p\left(\frac{-1}{2}\right)\textbf{v}, \] 
for $p\in \mathfrak{p}$ such that $\langle p,p \rangle=-2$.

Let $\text{Cliff}(\hat{\mathfrak{p}}_{\text{tw}})$ be the Clifford algebra attached to $\hat{\mathfrak{p}}_{\text{tw}}$. Define $\hat{\mathfrak{p}}_{\text{tw}}^{>}:=\mathfrak{p}[t]t$ and let $\langle \hat{\mathfrak{p}}_{\text{tw}}^{>} \rangle$ be the subalgebra of this Clifford algebra generated by $\hat{\mathfrak{p}}_{\text{tw}}^{>}$. Similarly, define $\hat{\mathfrak{p}}_{\text{tw}}^{-}:=\mathfrak{p}[t^{-1}]$ and $\langle \hat{\mathfrak{p}}_{\text{tw}}^{-} \rangle$. Take $\mathbb{C}\textbf{v}_{\text{tw}}$ to be a $\hat{\mathfrak{p}}_{\text{tw}}^{>}$ module such that $1\textbf{v}_{\text{tw}}=\textbf{v}_{\text{tw}}$ and $a(r)\textbf{v}_{\text{tw}}=0$ for $a\in \mathfrak{p}$ and $r>0$ . For $p\in \mathfrak{p}$ (as before) such that $\langle p,p \rangle=-2$, we have that $p(0)^2=1$ in $\text{Cliff}(\mathfrak{p})$. Define $\textbf{v}_{\text{tw}}^{+}:=(1+p(0))\textbf{v}_{\text{tw}}$ so that $p(0)\textbf{v}_{\text{tw}}^{+}=\textbf{v}_{\text{tw}}^{+}$.  Then we define  \begin{equation}A(\mathfrak{p})_{\text{tw}}^{+}:=\text{Cliff}(\hat{\mathfrak{p}}_{\text{tw}} )\otimes_{\langle \hat{\mathfrak{p}}_{\text{tw}}^{>}\rangle} \mathbb{C}\textbf{v}_{\text{tw}}^{+},\end{equation}

\noindent so that $A(\mathfrak{p})_{\text{tw}}^+$ is isomorphic  (as a $\langle \hat{\mathfrak{p}}_{\text{tw}}^{-} \rangle$-module) to $\bigwedge(p(-n)\mid n>0)\textbf{v}_{\text{tw}}^+$ (where $\bigwedge(x_1, x_2 \dots):= \bigwedge(\oplus_{i=1}^{\infty} \mathbb{C}x_i)$).

By the reconstruction theorem described in \cite{FrenkelBenzvi} we can see that $A(\mathfrak{p})_{\text{tw}}$ is a twisted module for $A(\mathfrak{p})$ with fields $Y_{\text{tw}}\colon A(\mathfrak{p}) \otimes A(\mathfrak{p})_{\text{tw}} \to A(\mathfrak{p})_{\text{tw}}(\!(z^{\frac{1}{2}})\!)$ such that $ Y_{\text{tw}}\left(u\left(\frac{-1}{2}\right)\textbf{v},z\right)=\sum\limits_{n \in \mathbb{Z}} u(n)z^{-n-\frac{1}{2}}$ for $u\in \mathfrak{p}$.  Since $A(\mathfrak{p})_{\text{tw}}^+$ is a submodule of $A(\mathfrak{p})_{\text{tw}}$ (generated by $\textbf{v}^+_{\text{tw}}$), it can be verified that $A(\mathfrak{p})_{\text{tw}}^+$ is a twisted module for $A(\mathfrak{p})$ so that the above map can be restricted to $A(\mathfrak{p})^+_{\text{tw}}$. In fact, ${A}(\mathfrak{p})$ is a canonically twisted ${A}(\mathfrak{p})$-module, by which we mean the twisted module for ${A}(\mathfrak{p})$ with respect to its parity involution (see also \cite{DMCconway}). 

Let $L_2(0)$ be the $L(0)$ operator for the Clifford module vertex superalgebra and $c_2$ its central charge. Then we can see that $\text{tr}\left(p(0)q^{L_2(0)-\frac{c_2}{24}}\mid A(\mathfrak{p})_{\text{tw}}^+\right)= \eta(\tau)$.  We would like a module with graded dimension equal to $\eta^4(\tau)$ so we will consider a tensor product of these $ A(\mathfrak{p})_{\text{tw}}^+$ (we have from \cite{FLH} that the tensor product of vertex superalgebras is naturally a vertex superalgebra and that the tensor product of twisted modules is a twisted module for the tensor product of vertex superalgebras).

 To do this, we define 
\begin{equation} \widetilde{A}({\mathfrak{p}})^{}:= {A}({\mathfrak{p}}_1) \otimes \cdots \otimes {A}(\mathfrak{p}_{4})  \end{equation}
and
\begin{equation} \widetilde{A}({\mathfrak{p}})^{}_{\text{tw}}:= {A}({\mathfrak{p}_{1} })_{\text{tw}}^+\otimes \cdots \otimes {A}(\mathfrak{p}_{4} )_{\text{tw}}^+  \end{equation}
where each $A(\mathfrak{p}_i)_{\text{tw}}^+$ is isomorphic to $\bigwedge(p_i(-n) \mid n>0)\textbf{v}_{\text{tw}}^+$. Then we can define

\begin{equation}\widetilde{Y}_{\text{tw}}\colon \widetilde{A}(\mathfrak{p})^{} \otimes \widetilde{A}(\mathfrak{p})^{+}_{\text{tw}} \to \widetilde{A}(\mathfrak{p})^{+}_{\text{tw}}(\!(z^{\frac{1}{2}})\!)\end{equation} where for $u_i \in {A}(\mathfrak{p}_i)$, \begin{equation}\begin{aligned} \widetilde{Y}_{\text{tw}}(u_1\otimes \cdots \otimes u_4,z)&=Y_{\text{tw}}(u_1,z) \otimes \cdots \otimes Y_{\text{tw}}(u_{4},z)\\&=\sum\limits_{n \in \mathbb{Z}^{4}} u_1(n_1)\otimes \cdots \otimes u_{4}(n_{4})z^{-n_1 \cdots -n_{4} -{2}}, \end{aligned}\end{equation}
with $n^{}=(n_1, \dots, n_{4})$,
and finally  \[ \widetilde{p}^{}(0):=p_1(0)\otimes \cdots \otimes p_{4}(0). \]

\noindent This completes the proof of the following lemma in which we record the second factor of equation \eqref{factors}.
\begin{Lemma}\label{clifflem}  \begin{align*}
\emph{tr}\left(\widetilde{p}^{}(0)q^{L_2(0)-\frac{c_2}{24}}\mid \widetilde{A}(\mathfrak{p})^{}_{\emph{tw}}\right)&= \eta^{4}(\tau).
\end{align*}
\end{Lemma}

For the third factor, we require a module with graded dimension function given by the expansion of $\dfrac{\eta^2(\tau)}{\theta_1^2(\tau,z)}$ in our usual domain ($\tau \in \mathbb{H}, z\in \mathbb{C}$ such that $0<-\text{Im}(z)<\text{Im}(\tau)$). To this end, we use a twisted module over a Weyl module vertex operator algebra. We follow Duncan and O'Desky for this construction \cite{DO}. In this setting, let $\mathfrak{b}$ be a $4$-dimensional vector space with a non-degenerate antisymmetric bilinear form. Let ${\hat{\mathfrak{b}}}_{}:=\mathfrak{b}[t,t^{-1}]t^{\frac{1}{2}}$ and $\hat{\mathfrak{b}}_{\text{tw}}:=\mathfrak{b}[t,t^{-1}]$, for $b \in \mathfrak{b}$ we write $b(r)$ for $bt^r$ with the bilinear form extended as in the case of the Clifford module vertex operator algebra.

Let $\text{Weyl}(\hat{\mathfrak{b}})$ be the Weyl algebra associated to $\hat{\mathfrak{b}}$ and its antisymmetric bilinear form. Define  $\hat{\mathfrak{b}}^+:=\mathfrak{b}[t]t^{\frac{1}{2}}$ and  $\hat{\mathfrak{b}}^-:=\mathfrak{b}[t^{-1}]t^{\frac{1}{2}}$ so that $\hat{\mathfrak{b}}=\hat{\mathfrak{b}}^+\bigoplus \hat{\mathfrak{b}}^-$ is a polarization for the antisymmetric bilinear form so that $\hat{\mathfrak{b}}^{\pm}$ is isotropic. Let $\mathbb{C}\textbf{v}$ be the unique unital $\langle \hat{\mathfrak{b}}^+\rangle$-module such that $b\textbf{v}=0$ for every $b \in \hat{\mathfrak{b}}^+$.%, and $Y(b\left(\frac{-1}{2}\right)\textbf{v},z)=\sum\limits_{n\in \mathbb{Z}} b(n+\frac{1}{2})z^{-n-1}$ for $b\in \mathfrak{b}$

We define the Weyl module vertex algebra associated to $\mathfrak{b}$ and the antisymmetric bilinear form to be the unique vertex superalgebra structure on \begin{equation} \mathdutchcal{W}(\mathfrak{b}):= \text{Weyl}(\hat{\mathfrak{b}})\otimes_{\langle {\hat{\mathfrak{b}}}^+\rangle} \mathbb{C}\textbf{v} \end{equation} such that $Y(b\left(\frac{-1}{2}\right)\textbf{v},z)=\sum\limits_{n\in \mathbb{Z}} b(n+\frac{1}{2})z^{-n-1}$ for $b\in \mathfrak{b}$.

Let $\{ b_i^\pm\}$ be a basis for $\mathfrak{b}^\pm$ such that $\langle \langle b_i^\mp,b_j^\pm\rangle=\pm\delta_{ij}$ where $\langle \langle \cdot, \cdot \rangle$ is the antisymmetric bilinear form on $\mathfrak{b}$. Then define \[ \omega:=\dfrac{1}{2}\sum_i \left( b_i^+\left(\frac{-3}{2}\right) b_i^-\left(\frac{-1}{2}\right)-  b_i^+\left(\frac{-1}{2}\right) b_i^-\left(\frac{-3}{2}\right) \right)\textbf{v}. \]

Then equipped with Virasoro element $\omega$, $ \mathdutchcal{W}(\mathfrak{b})$ has the structure of a Weyl module vertex operator algebra.%such that $\textbf{v}$ is the vacuum, $\mathbb{C}\textbf{v}$ is the unique unital $\langle \hat{\mathfrak{b}}^+\rangle$-module such that $b\textbf{v}=0$ for every $b \in \hat{\mathfrak{b}}^+$, and $Y(b\left(\frac{-1}{2}\right)\textbf{v},z)=\sum\limits_{n\in \mathbb{Z}} b(n+\frac{1}{2})z^{-n-1}$ for $b\in \mathfrak{b}$.

%Let $\{ b_i^\pm\}$ be a basis for $\mathfrak{b}^\pm$ such that $\langle \langle b_i,b_j\rangle=\pm\delta_{ij}$ where $\langle \langle \cdot, \cdot \rangle$ is the antisymmetric bilinear form on $\mathfrak{b}$. Then define \[ \omega:=\dfrac{1}{2}\sum_i \left( b_i^+\left(\frac{-3}{2}\right) b_i^-\left(\frac{-1}{2}\right)-  b_i^+\left(\frac{-1}{2}\right) b_i^-\left(\frac{-3}{2}\right) \right)\textbf{v}, \]
 %and this will be our choice of Virasoro element giving $\mathdutchcal{W}(\mathfrak{b})$ the structure of a vertex operator superalgebra. 
 
 Similarly, for $\mathfrak{b}^+$ defined to be the span of $\{ b_i^+\}$ and $\mathfrak{b}^-$ defined to be the span of $\{ b_i^-\}$, we can define ${\hat{\mathfrak{b}}}^+_{\text{tw}}:=\mathfrak{b}^+\oplus t \mathfrak{b}[t]$ and ${\hat{\mathfrak{b}}}^-_{\text{tw}}:=\mathfrak{b}^-\oplus t^{-1}\mathfrak{b}[t^{-1}]$.  Let $\mathbb{C}\textbf{v}_{\text{tw}}$ be the unique unital $\langle \hat{\mathfrak{b}}^+_{\text{tw}}\rangle$-module such that $b\textbf{v}_{\text{tw}}=0$ for every $b \in \hat{\mathfrak{b}}^+_{\text{tw}}$. Then $\mathdutchcal{W}(\mathfrak{b})_{\text{tw}}$ has the structure of a twisted $\mathdutchcal{W}(\mathfrak{b})$-module:
   \begin{equation}\mathdutchcal{W}(\mathfrak{b})_{\text{tw}}:= \text{Weyl}(\hat{\mathfrak{b}}_{\text{tw}})\otimes_{\langle {\hat{\mathfrak{b}}}^+_{\text{tw}}\rangle} \mathbb{C}\textbf{v}_{\text{tw}}\end{equation}
   such that $Y_{\text{tw}}(b\left(\frac{-1}{2}\right),z)=\sum\limits_{n\in \mathbb{Z}} b(n)z^{-n-\frac{1}{2}}$ for $b\in \mathfrak{b}$. $\mathdutchcal{W}(\mathfrak{b})_{\text{tw}}$ is the unique (up to equivalence) irreducible canonically twisted $\mathdutchcal{W}(\mathfrak{b})$-module.%By canonically twisted $\mathdutchcal{W}(\mathfrak{b})$-module we mean the twisted module for $\mathdutchcal{W}(\mathfrak{b})$ with respect to its parity involution. We refer to \cite{DMCconway} for the complete definition of a canonically twisted $V$-module. 
 
Denote the central charge of $\mathdutchcal{W}(\mathfrak{b})$ by $c_3$. We also denote by $L_3(n)$ the coefficient of $z^{-n-2}$ in $Y(\omega, z)$ or $Y(\omega, z)_{\text{tw}}$. Letting 

\[ \jmath:=\sum_i b_i^+\left(\frac{-1}{2}\right) b_i^-\left(\frac{-1}{2}\right)\textbf{v}, \]
we denote by $J_3(n)$ the  coefficient of $z^{-n-1}$ in $Y(\jmath, z)$ or $Y(\jmath, z)_{\text{tw}}$. Thus we have a bigrading on both $\mathdutchcal{W}(\mathfrak{b})$ and $\mathdutchcal{W}(\mathfrak{b})_{\text{tw}}$. We focus on the latter, $\mathdutchcal{W}(\mathfrak{b})_{\text{tw}}$, which has bigraded dimension as follows:
\begin{equation}\label{wbtwtrace}
\text{tr}\left(y^{J_3(0)}q^{L_3(0)-\frac{c_3}{24}}\mid \mathdutchcal{W}(\mathfrak{b})_{\text{tw}}\right)=y^{-1}q^{-\frac{1}{6}} \prod\limits_{n>0} (1-y^{-1}q^{n-1})^{-2}(1-yq^n)^{-2}.
\end{equation}
\begin{remark} For the above equation to make sense we should expand the right hand side in the domain to which we have restricted, $0<-\text{Im}(z)<\text{Im}(\tau)$. In other words, each factor of the form $\dfrac{1}{(1-X)}$ should be interpreted as  $\sum\limits_{n\geq0}X^n$.
\end{remark}

\begin{Lemma}\label{W3} \[-\emph{tr}\left(y^{J_3(0)}q^{L_3(0)-\frac{c_3}{24}}\mid \mathdutchcal{W}(\mathfrak{b})_{\emph{tw}}\right)=\dfrac{\eta^2(\tau)}{\theta_1^2(\tau,z)}\]

\end{Lemma}
\begin{proof}% We first recall equation B.2 of \cite{DO} for ${\theta_1(\tau,z)}$:

%\begin{equation} {\theta_1(\tau,z)}=-iy^{\frac{1}{2}}q^{\frac{1}{8}} \prod\limits_{n>0} (1-y^{-1}q^{n-1})(1-yq^n)(1-q^n),
%\end{equation}
The equation for ${\theta_1(\tau,z)}$ (see Section \ref{thefunctions}) implies the following equation for $\dfrac{1}{\theta_1^2(\tau,z)}$:

\begin{equation} \dfrac{1}{\theta_1^2(\tau,z)}=-y^{-1}q^{-\frac{1}{4}} \prod\limits_{n>0} (1-y^{-1}q^{n-1})^{-2}(1-yq^n)^{-2}(1-q^n)^{-2},
\end{equation}

\noindent Noting that $\eta^2(\tau)=q^{\frac{1}{12}}\prod\limits_{n>0} (1-q^n)^2$ we see that
\begin{equation} \dfrac{\eta^2(\tau)}{\theta_1^2(\tau,z)}=-y^{-1}q^{-\frac{1}{6}} \prod\limits_{n>0} (1-y^{-1}q^{n-1})^{-2}(1-yq^n)^{-2},
\end{equation}
and the expansion of this in our specified domain is equal to the graded dimension of $\mathdutchcal{W}(\mathfrak{b})_{\text{tw}}$ (see \eqref{wbtwtrace}). \end{proof}

 To recover the first factor of \eqref{factors} we need a module with graded dimension function $\phi_g(\tau,z)$. For this we use the canonically twisted $V^{s\natural}$-module, $V^{s\natural}_{\text{tw}}$, where $V^{s\natural}$ is the unique self-dual, rational, $C_2$-cofinite vertex operator superalgebra of CFT type with central charge $12$ such that $L(0)u=\frac{1}{2}u$ for $u\in V^{s\natural}$ implies $u=0$ (cf. Theorem 5.15 \cite{JohnDuke} and Theorem 4.5 \cite{DMCconway}). 
 
 For full details of the construction of $V^{s\natural}_{\text{tw}}$, we refer the reader to Duncan and Mack-Crane \cite{DMCconway}. In what follows we give a brief summary. To define $V^{s\natural}_{\text{tw}}$ here, we start with the construction of Clifford algebra modules (see above), but this time instead of starting with a one-dimensional complex vector space, we take $\mathfrak{a}$ to be a $24$-dimensional complex vector space with a non-degenerate symmetric bilinear form. Let ${\hat{\mathfrak{a}}}_{}:=\mathfrak{a}[t,t^{-1}]t^{\frac{1}{2}}$ and $\hat{\mathfrak{a}}_{\text{tw}}:=\mathfrak{a}[t,t^{-1}]$, for $a \in \mathfrak{a}$ we write $a(r)$ for $at^r$ with the bilinear form extended as before. We define a polarization $\hat{\mathfrak{a}}=\hat{\mathfrak{a}}^+\oplus \hat{\mathfrak{a}}^-$ of $\hat{\mathfrak{a}}$ by setting ${\hat{\mathfrak{a}}}^+:=\mathfrak{a}[t]t^{\frac{1}{2}}$ and ${\hat{\mathfrak{a}}}^-:=\mathfrak{a}[t^{-1}]t^{\frac{-1}{2}}$. Let $\mathbb{C}\textbf{v}$ be the unique unital $\langle \hat{\mathfrak{a}}^+ \rangle$-module such that $a\textbf{v}= 0$ for every $a\in \hat{\mathfrak{a}}^+$.
Then we can define $A(\mathfrak{a})$ to be the $\text{Cliff}(\hat{\mathfrak{a}})$-module:\begin{equation}A(\mathfrak{a}):=\text{Cliff}(\hat{\mathfrak{a}})\otimes_{\langle \hat{\mathfrak{a}}^+ \rangle} \mathbb{C}\mathbf{v},\end{equation} where as $\langle \hat{\mathfrak{a}}^- \rangle$-modules, we have the isomorphism $A(\mathfrak{a})\simeq \bigwedge(\hat{\mathfrak{a}}^- )\mathbf{v}$.

$A(\mathfrak{a})$ has the structure of a vertex superalgebra such that $Y(a\left(-\frac{1}{2}\right)\textbf{v},z)=\sum\limits_{n\in \mathbb{Z}} a(n+\frac{1}{2})z^{-n-1}$ for $a\in \mathfrak{a}$. The super space structure $A(\mathfrak{a})=A(\mathfrak{a})^0 \oplus A(\mathfrak{a})^1$ is given by the parity decomposition on $\bigwedge(\hat{\mathfrak{a}}^- )\mathbf{v}$.
 
 For $\{ e_i \}$ an orthonormal basis for $\mathfrak{a}$, the Virasoro element \[ \omega=-\dfrac{1}{4}\sum_{i=1}^{\text{dim}{\mathfrak{a}}}  e_i\left(-\frac{3}{2}\right) e_i\left(-\frac{1}{2}\right) \textbf{v}, \] gives $A(\mathfrak{a})$ the structure of a vertex operator superalgebra. %such that $Y(a\left(-\frac{1}{2}\right)\textbf{v},z)=\sum\limits_{n\in \mathbb{Z}} a(n+\frac{1}{2})z^{-n-1}$ for $a\in \mathfrak{a}$. The super space structure $A(\mathfrak{a})=A(\mathfrak{a})^0 \oplus A(\mathfrak{a})^1$ is given by the parity decomposition on $\bigwedge(\hat{\mathfrak{a}}^- )\mathbf{v}$. %For $\{ e_i \}$ an orthonormal basis for $\mathfrak{a}$, the Virasoro element \[ \omega=-\dfrac{1}{4}\sum_{i=1}^{\text{dim}{\mathfrak{a}}}  e_i\left(-\frac{3}{2}\right) e_i\left(-\frac{1}{2}\right) \textbf{v}, \] gives $A(\mathfrak{a})$ the structure of a vertex operator superalgebra.
 
 Similarly, for $\mathfrak{a}=\mathfrak{a}^+ \oplus \mathfrak{a}^-$ a polarization of $\mathfrak{a}$ with respect to its non-degenerate symmetric bilinear form, we can define ${\hat{\mathfrak{a}}}^+_{\text{tw}}:=\mathfrak{a}^+\oplus t \mathfrak{a}[t]$ and ${\hat{\mathfrak{a}}}^-_{\text{tw}}:=\mathfrak{a}^-\oplus t^{-1}\mathfrak{a}[t^{-1}]$. Let $\mathbb{C}\textbf{v}_{\text{tw}}$ be the unique unital $\langle \hat{\mathfrak{a}}^+\rangle$-module such that $u\textbf{v}_{\text{tw}}=0$ for $u\in  \hat{\mathfrak{a}}^+$. Then \begin{equation}A(\mathfrak{a})_{\text{tw}}:= \text{Cliff}(\hat{\mathfrak{a}}_{\text{tw}})\otimes_{\langle {\hat{\mathfrak{a}}}^+_{\text{tw}}\rangle} \mathbb{C}\textbf{v}_{\text{tw}} \end{equation} has the  structure of a twisted $A(\mathfrak{a})$-module such that $Y_{\text{tw}}(a\left(-\frac{1}{2}\right),z)=\sum\limits_{n\in \mathbb{Z}} a(n)z^{-n-\frac{1}{2}}$ for $a\in \mathfrak{a}$. This is the unique (up to equivalence) irreducible canonically twisted $A(\mathfrak{a})$-module. We again have the isomorphism $A(\mathfrak{a})_{\text{tw}}\simeq \bigwedge(\hat{\mathfrak{a}}_{\text{tw}}^- )\mathbf{v}_{\text{tw}}$ as $\langle \hat{\mathfrak{a}}_{\text{tw}}^-\rangle$-modules.

We will also define a decomposition of $A(\mathfrak{a})_{\text{tw}}$. For this we first define  $\mathfrak{z}\in \text{Spin}(\mathfrak{a})$ to be the unique lift of $-\text{Id}_{\mathfrak{a}} \in SO(\mathfrak{a})$ to $\text{Spin}(\mathfrak{a})$ such that %$\mathfrak{z}(\cdot)=-\text{Id}_{\mathfrak{a}}$ and 
$\mathfrak{z}\mathbf{v}_{\text{tw}}=\mathbf{v}_{\text{tw}}$. The element $\mathfrak{z}$ acts on $A(\mathfrak{a})_{\text{tw}}$ with order two and we denote by $A(\mathfrak{a})^j_{\text{tw}}$ the eigenspace for this action with eigenvalue $(-1)^j$. We can decompose $A(\mathfrak{a})_{\text{tw}}$ into eigenspaces $A(\mathfrak{a})_{\text{tw}}=A(\mathfrak{a})^0_{\text{tw}}\oplus A(\mathfrak{a})^1_{\text{tw}}$. For more on the lift to the spin group, we refer the reader to \cite{DMCconway}.
For the rest of the construction we refer to \cite{DMCderived}. Taking $\mathfrak{a}=\Lambda \otimes_{\mathbb{Z}}\mathbb{C}$, we define \begin{equation} V^{s\natural}=A(\mathfrak{a})^0 \oplus A(\mathfrak{a})^1_{\text{tw}}, \hspace{3mm} V^{s\natural}_{\text{tw}}=A(\mathfrak{a})^1 \oplus A(\mathfrak{a})^0_{\text{tw}}.
\end{equation}
Denote the central charge of $V^{s\natural}$ by $c_1$, and denote by $L_1(n)$ the coefficient of $z^{-n-2}$ in $Y(\omega, z)$ or $Y(\omega, z)_{\text{tw}}$.

We will define an additional operator $J_1(n)$ in order to define a bi-grading on $V^{s\natural}_{\text{tw}}$. To define this operator, first let $\Pi$ be a $4$-dimensional subspace of $\Lambda \otimes_{\mathbb{Z}}\mathbb{C}$ and let $\{x, y, z, w\}$ be an orthonormal basis for $\Pi$. We then define $a_1^{\pm}=\frac{1}{\sqrt{2}}(x\pm {i}y)$ and $a_2^{\pm}=\frac{1}{\sqrt{2}}(z\pm {i}w)$ so that $\langle a_1^{\pm}, a_1^{\mp}\rangle=\langle a_2^{\pm}, a_2^{\mp}\rangle=1$.
Then we let
\[ \jmath:=\frac{1}{2}a_1^-\left(-\frac{1}{2}\right) a_1^+\left(-\frac{1}{2}\right)\mathbf{v}+ \frac{1}{2} a_2^-\left(-\frac{1}{2}\right) a_2^+\left(-\frac{1}{2}\right)\mathbf{v}, \]
and we denote by $J_1(n)$ the coefficient of $z^{-n-1}$ in $Y(\jmath, z)$ or $Y(\jmath, z)_{\text{tw}}$. 
The operators $L_1(0)$ and $J_1(0)$ then equip $V^{s\natural}_{\text{tw}}$ with a bigrading as follows:

\begin{equation}(V^{s\natural}_{\text{tw}})_{{n,r}}=\{v\in V^{s\natural}_{\text{tw}} \mid (L_1(0)-\tfrac{c}{24})v=nv, J_1(0)v=rv \}.
\end{equation}
Taking $\mathfrak{a}=\Lambda \otimes_{\mathbb{Z}}\mathbb{C}$ allows us to identify $Co_0$ (and therefore $M_{24}$) with a subgroup of $SO(\mathfrak{a})$. By Proposition $3.1$ of \cite{DMCconway}, for any subgroup $G$ of $SO(\mathfrak{a})$ which is isomorphic to $Co_0$, there exists a unique lift of $G$ to $\text{Spin}(\mathfrak{a})$ such that the non-trivial central element is $\mathfrak{z}$. We denote this lift by $\widehat{G}$ and for $g\in G$, we denote the lift of $g$ to $\widehat{G}$ by $\widehat{g}$. The spin group acts naturally on $V^{s\natural}$ and $V^{s\natural}_{\text{tw}}$ so we can now state the following lemma: 
\begin{Lemma}\label{phi3}For $g\in G$ fixing a $4$-space in the $24$-dimensional permutation representation of $M_{24}$, we have
\begin{equation}\phi_g(\tau,z)=-\emph{tr}\left(\mathfrak{z}\widehat{g}y^{J_1(0)}q^{L_1(0)-\frac{c_1}{24}}\mid V^{s\natural}_{\emph{tw}}\right).\end{equation} \end{Lemma}\
Define the operators $L(0):=L_1(0)+L_2(0)+L_3(0)$ and $J(0):=J_1(0)+J_3(0)$ and the central charge $c:=c_1+c_2+c_3$. Combining Lemmas \ref{clifflem}, \ref{W3}, and \ref{phi3} we can state the following theorem:
\begin{Theorem} $\widetilde{A}(\mathfrak{p})^{}_{\emph{tw}}\otimes \mathdutchcal{W}(\mathfrak{b})_{\emph{tw}} \otimes V^{s\natural}_{\emph{tw}}$ is an infinite dimensional, bigraded, virtual module with trace functions as follows:
\begin{equation}
\emph{tr}\left(\widehat{g}\mathfrak{z}\widetilde{p}(0)y^{J(0)}q^{L(0)-\frac{c}{24}}\mid \widetilde{A}(\mathfrak{p})^{}_{\emph{tw}}\otimes \mathdutchcal{W}(\mathfrak{b})_{\emph{tw}} \otimes V^{s\natural}_{\emph{tw}}\right)=\widetilde{M_g}(\tau,z).\end{equation}
\end{Theorem}
 % but in order to do this we first we define the spin group of $\mathfrak{a}$, Spin($\mathfrak{a}$).% Spin($\mathfrak{a}$) is defined to be the set of even, invertible elements $x\in \text{Cliff}(\mathfrak{a})$ with $\alpha(x)x=1$ such that $xux^{-1}\in \mathfrak{a}$ whenever $u\in \mathfrak{a}$.
 This gives a module construction for any subgroup $G$ of $M_{24}$ for which the $24$-dimensional permutation representation of $M_{24}$ restricted to $G$ fixes at least a four dimensional space.
 
  \begin{Example}
 $\widetilde{A}(\mathfrak{p})^{}_{\text{tw}}\otimes \mathdutchcal{W}(\mathfrak{b})_{\text{tw}} \otimes V^{s\natural}_{\text{tw}}$ is a (virtual) module for the group $\bf{L_3(4) \simeq M_{21}}$, one of the simple subgroups of $M_{24}$. One can see via the following fusion of conjugacy classes \[\{ 1A, 2A, 3A, 4B, 4B, 4B, 5A, 5A, 7A, 7B \}\] that the $24$-dimensional representation of $M_{24}$ restricts to $\bf{L_3(4)}$ as $4\psi_1+1\psi_2$ (where $\psi_i$ are irreducible representations of $\bf{L_3(4)}$ and $\psi_1$ is the trivial representation). In particular, we see the permutation representation restricted to $\bf{L_3(4)}$ fixes a four dimensional space.
 \end{Example}
 
\section{Module construction II} \label{mod2}
The construction described in the previous section does not apply in cases where the restriction of the $24$-dimensional permutation representation to a subgroup of $M_{24}$ does not fix a $4$-space. In what follows we give a similar module construction for such subgroups. However, for these subgroups, we still require that each element of the subgroup fixes a $4$-space. Note that this is a weaker requirement than asking that the subgroup itself fixes a $4$-space, because not every element of the subgroup necessarily fixes the same $4$-space. For this construction we apply a method of Anagiannis, Cheng, Harrison \cite{ACH}. In our context, the idea of the method is to view the theta quotients and the eta quotients in $\phi_g(\tau,z)$ (the graded trace functions of $V^{s\natural}_{\text{tw}}$) as coming from dimensions of different spaces (see \eqref{phig}). 

We begin by constructing another module $T$ which we show has the same the trace functions as those from $V^{s\natural}_{\text{tw}}$ (recall that $V^{s\natural}_{\text{tw}}= A(\mathfrak{a})^1\oplus A(\mathfrak{a})_{\text{tw}}^0$). 

We define $\mathfrak{f}:=\mathbb{C}^4$ and equip it with both a non-degenerate symmetric bilinear form $\langle \cdot ,\cdot \rangle$ and a non-degenerate antisymmetric bilinear form $\langle \langle \cdot ,\cdot \rangle$. For convenience we make the choice in such a way that a decomposition $\mathfrak{f}= \mathfrak{f}^+ \bigoplus \mathfrak{f}^-$ serves as a polarization for both bilinear forms. Then we may define
\begin{align}B=A(\mathfrak{f})\otimes \mathdutchcal{W}(\mathfrak{f}) & \text{ and } B_{\text{tw}}=A(\mathfrak{f})_{\text{tw}}\otimes \mathdutchcal{W}(\mathfrak{f})_{\text{tw}},\end{align} where $A(\mathfrak{f})$ and $\mathdutchcal{W}(\mathfrak{f})$ are defined, as before, to be a Cliff($\hat{\mathfrak{f}}$)-module and a $\text{Weyl}(\hat{\mathfrak{f}})$-module associated to $\mathfrak{f}$, each endowed with a vertex superalgebra (resp. vertex algbebra) structure. 

For $A(\mathfrak{f})$ we let $\{ f_i^\pm\}$ be a basis for $\mathfrak{f}^\pm$ such that $\langle f_i^\mp, f_j^\pm\rangle=\delta_{ij}$ where $\langle \cdot ,\cdot \rangle$ is the non-degenerate symmetric bilinear form on $\mathfrak{f}$. We can then define the elements $\jmath:=\sum_i f_i^+({-\frac{1}{2}})f_i^-({-\frac{1}{2}})\mathbf{v}$ and $\omega:=\dfrac{1}{2}\sum_i \left( f_i^+\left(\frac{-3}{2}\right) f_i^-\left(\frac{-1}{2}\right)-  f_i^+\left(\frac{-1}{2}\right) f_i^-\left(\frac{-3}{2}\right) \right)\textbf{v}$ and we denote the corresponding operators by $J_{11}(0)$ and $L_{11}(0)$ and the central charge by $c_{11}$.

Similarly, for $\mathdutchcal{W}(\mathfrak{f})$ we assume that the antisymmetric bilinear form $\langle \langle \cdot ,\cdot \rangle$ on $\mathfrak{f}$ is chosen so that $\langle\langle f_i^\mp, f_j^\pm\rangle=\pm\delta_{ij}$. We define the elements $\jmath:=\sum_i f_i^+({\frac{-1}{2}})f_i^-({\frac{-1}{2}})\mathbf{v}$ and \newline $\omega:=\dfrac{1}{2}\sum_i \left( f_i^+\left(\frac{-3}{2}\right) f_i^-\left(\frac{-1}{2}\right)-  f_i^+\left(\frac{-1}{2}\right) f_i^-\left(\frac{-3}{2}\right) \right)\textbf{v}$ and denote the corresponding operators $J_{12}(0)$ and $L_{12}(0)$ and the central charge $c_{12}$.

Lastly, $A(\mathfrak{a})$, for $\mathfrak{a}=\Lambda \otimes_{\mathbb{Z}} \mathbb{C}$, along with the conformal vector associated to it has already been defined in the previous section, but here we will rename the $L(0)$ operator and the central charge associated to $A(\mathfrak{a})$ as $L_{13}(0)$ and $c_{13}$, respectively. We do not define the element $\jmath$ or the operator $J(0)$ for $A(\mathfrak{a})$ because we are no longer assuming that all $g$ in our subgroup fix a single $4$-space in $\mathfrak{a}$.

We can now make the definition: \[T:=(B \otimes A(\mathfrak{a}))^1 \oplus (B \otimes A(\mathfrak{a}))^0_{\text{tw}}\]
and we can compute the trace of $g \in M_{24}$ (for $g$ that fix a $4$-space, and are in the allowed conjugacy classes) acting on $T$.

We let $\lambda_i^{\pm1}$ be the eigenvalues for $g$ acting on $\mathfrak{a}$. Since we are restricting to $g\in M_{24}$ fixing a $4$-space of $\mathfrak{a}$, we can assume that for two $i$ we have $\lambda_i^{\pm1}=1$. We also define $\nu_i$ to be square roots of the $\lambda_i$ and $\nu=\prod\limits_{i=1}^{12} \nu_i$. Before we can compute the trace of $\widehat{g}\mathfrak{z}y^{J_1(0)}q^{L_1(0)-\frac{c_1}{24}}$ on $T$ we require a few more definitions.%With this, we can compute the trace of $\widehat{g}q^{L_1(0)-c_1/24}$ and of $\widehat{g}\mathfrak{z}q^{L_1(0)-c/24}$ on $A(\mathfrak{a})$ and $A(\mathfrak{a})_{\text{tw}}$, to determine the contribution of trace of $\widehat{g}\mathfrak{z}q^{L_1(0)-c/24}$ on $A(\mathfrak{a})^1$ and $A(\mathfrak{a})^0_{\text{tw}}$. Note that $\widehat{g}\mathfrak{z}$ acts trivially on $B$, so with this information we will have determined the trace of $\widehat{g}\mathfrak{z}q^{L_1(0)-c/24}$ on all of $W$.

We recall the product formulas of the Jacobi theta functions \begin{equation} \theta_1(\tau,z):=-i q^{\frac{1}{8}}y^{\frac{1}{2}}\prod\limits_{n>0} (1-y^{-1}q^{n-1})(1-yq^n)(1-q^n), \end{equation}
 \begin{equation}\theta_2(\tau,z):= q^{\frac{1}{8}}y^{\frac{1}{2}}\prod\limits_{n>0} (1+y^{-1}q^{n-1})(1+yq^n)(1-q^n), \end{equation}
  \begin{equation}\theta_3(\tau,z):=\prod\limits_{n>0} (1+y^{-1}q^{n-1/2})(1+yq^{n-1/2})(1-q^n), \end{equation}
 and
 \begin{equation} \theta_4(\tau,z):=\prod\limits_{n>0} (1-y^{-1}q^{n-1/2})(1-yq^{n-1/2})(1-q^n). \end{equation}
 We then recall the definition 
 \begin{equation}\eta_g(\tau):=q\prod_{n>0}\prod\limits_{i=1}^{12} (1-\lambda_i^{-1}q^n)(1-\lambda_iq^n), \end{equation}
and note that
\begin{equation}\dfrac{\eta_g({\tau}/{2})}{\eta_g(\tau)}=q^{-\frac{1}{2}}\prod_{n>0}\prod_{i=1}^{12}(1-\lambda_i^{-1}q^{n-\frac{1}{2}})(1-\lambda_iq^{n-\frac{1}{2}}). \end{equation}

\noindent We also define \begin{equation} C_g=\nu\prod_{i=1}^{12} (1-\lambda_i^{-1})\end{equation} and \begin{equation} D_g=\nu^'\prod_{i=1}^{10} (1-\lambda_i^{-1}),\end{equation}
where $\nu^'$ is the product $\prod_{i=1}^{10}\nu_i$ (where we choose the labelling so that $\lambda^\pm_i=1$ for $i=11$ and $i=12$).

With these definitions, we are able to give the following explicit expression for $\phi_g(\tau,z)$ from Proposition 9.2 of \cite{DMCderived}:
\begin{equation}\label{phig}
\phi_g(\tau,z)=-\dfrac{1}{2}\left(\dfrac{\theta_4^2(\tau,z)}{\theta_4^2(\tau,0)}\dfrac{\eta_{-g}(\tau/2)}{\eta_{-g}(\tau)}-\dfrac{\theta_3^2(\tau,z)}{\theta_3^2(\tau,0)}\dfrac{\eta_{-g}(\tau/2)}{\eta_{-g}(\tau)}\right)+\dfrac{1}{2}\left(D_g \eta_g(\tau)\dfrac{\theta_1^2(\tau,z)}{\eta^6(\tau)}+C_{-g}\eta_{-g}(\tau)\dfrac{\theta^2_2(\tau,z)}{\theta^2_2(\tau,0)}\right).
\end{equation}

\noindent In the next few lemmas, we show that the trace of $\widehat{g}\mathfrak{z}y^{J_1(0)}q^{L_1(0)-c_1/24}$ on $T$ is equal to $\phi_g(\tau,z)$.

\begin{Lemma} Let $g\in M_{24}$ such that $\widehat{g}$ fixes a $4$-dimensional space of $\mathfrak{a}$. Let $\mathfrak{z}$ denote the parity involution, let $c_1$ be the central charge, let $L_1(0)$ and $J_{11}(0)$ and $J_{12}(0)$ be operators as before,  then
\begin{equation}\label{qW}  -\lim\limits_{\gamma \to -1} \emph{tr}\left(\widehat{g}\mathfrak{z}y^{J_{11}(0)}\gamma^{J_{12}(0)}q^{L_1(0)-\frac{c_1}{24}}\mid (B \otimes A(\mathfrak{a}))^1\right)=-\dfrac{1}{2}\left(\dfrac{\theta_4^2(\tau,z)}{\theta_4^2(\tau,0)}\dfrac{\eta_{-g}(\tau/2)}{\eta_{-g}(\tau)}-\dfrac{\theta_3^2(\tau,z)}{\theta_3^2(\tau,0)}\dfrac{\eta_{-g}(\tau/2)}{\eta_{-g}(\tau)}\right).
\end{equation}
\end{Lemma}

\begin{proof} We begin by recalling the projection operator $P^1(g)=\frac{1}{2}(g-\mathfrak{z}g)$. This will allow us to compute the graded trace \eqref{qW} on $(B \otimes A(\mathfrak{a}))^1$ by using the traces of $\widehat{g}\mathfrak{z}y^{J_1(0)}q^{L_1(0)-\frac{c_1}{24}}$ and $\widehat{g}y^{J_1(0)}q^{L_1(0)-\frac{c_1}{24}}$ on $B \otimes A(\mathfrak{a})$.

We will first compute the traces of $\widehat{g}q^{L_1(0)-\frac{c}{24}}$ and $\widehat{g}\mathfrak{z}q^{L_1(0)-\frac{c}{24}}$ on $A(\mathfrak{a})$. The traces are as follows:
\begin{equation}\label{ga}\text{tr}\left(\widehat{g}q^{L_{13}(0)-\frac{c_{13}}{24}}\mid A(\mathfrak{a})\right)= q^{-\frac{1}{2}}\prod\limits_{n>0}(1+q^{n-\frac{1}{2}})^4\prod\limits_{i=1}^{10} (1+\lambda_iq^{n-\frac{1}{2}})(1+\lambda_i^{-1}q^{n-\frac{1}{2}}),
\end{equation}
and 
\begin{equation}\label{zga}\text{tr}\left(\widehat{g}\mathfrak{z}q^{L_{13}(0)-\frac{c_{13}}{24}}\mid A(\mathfrak{a})\right)= q^{-\frac{1}{2}}\prod\limits_{n>0}(1-q^{n-\frac{1}{2}})^4\prod\limits_{i=1}^{10} (1-\lambda_iq^{n-\frac{1}{2}})(1-\lambda_i^{-1}q^{n-\frac{1}{2}}).
\end{equation}

Note that $\widehat{g}$ acts trivially on the components of $B$ because $\mathfrak{f}$ is fixed by $\widehat{g}$. So we compute the traces on the components of $B$ as follows: \begin{equation}\label{af} \text{tr}\left(y^{J_{11}(0)}q^{L_{11}(0)-\frac{c_{11}}{24}} \mid A(\mathfrak{f})\right)=q^{-\frac{1}{12}}\prod\limits_{n>0} (1+y^{-1}q^{n-\frac{1}{2}})^2(1+yq^{n-\frac{1}{2}})^2,
\end{equation}

\begin{equation}\label{zaf} \text{tr}\left(\mathfrak{z}y^{J_{11}(0)}q^{L_{11}(0)-\frac{c_{11}}{24}} \mid A(\mathfrak{f})\right)=q^{-\frac{1}{12}}\prod\limits_{n>0} (1-y^{-1}q^{n-\frac{1}{2}})^2(1-yq^{n-\frac{1}{2}})^2,
\end{equation}

\begin{equation}\label{wf} \text{tr}\left(\gamma^{J_{12}(0)}q^{L_{12}(0)-\frac{c_{12}}{24}} \mid \mathdutchcal{W}(\mathfrak{f})\right)=q^{\frac{1}{12}}\prod\limits_{n>0} (1-\gamma^{-1}q^{n-\frac{1}{2}})^{-2}(1-\gamma q^{n-\frac{1}{2}})^{-2},
\end{equation}
and 
\begin{equation}\label{zwf} \text{tr}\left(\mathfrak{z}\gamma^{J_{12}(0)}q^{L_{12}(0)-\frac{c_{12}}{24}} \mid \mathdutchcal{W}(\mathfrak{f})\right)=q^{\frac{1}{12}}\prod\limits_{n>0} (1+\gamma^{-1}q^{n-\frac{1}{2}})^{-2}(1+\gamma q^{n-\frac{1}{2}})^{-2}.
\end{equation}

Define the operators $L_1(0):=L_{11}(0)+L_{12}(0)+L_{13}(0)$ and the central charge $c_1:=c_{11}+c_{12}+c_{13}$.

We can then combine \cref{ga,af,wf} and take the limit as $\gamma \to -1$ to compute
 \begin{equation}\begin{aligned}\text{tr}\Big(\widehat{g}\mathfrak{z}q^{L_1(0)-\frac{c_1}{24}}y^{J_{11}(0)}\gamma^{J_{12}(0)}&\mid B\otimes A(\mathfrak{a})\Big)\\
 &=q^{-\frac{1}{2}}\prod\limits_{n>0} (1+y^{-1}q^{n-\frac{1}{2}})^2 (1+yq^{n-\frac{1}{2}})^2 \prod\limits_{i=1}^{10} (1+\lambda_i^{-1}q^{n-\frac{1}{2}})(1+\lambda_iq^{n-\frac{1}{2}})\\
&=\dfrac{\theta_3^2(\tau,z)}{\theta_3^2(\tau,0)}\dfrac{\eta_{-g}(\tau/2)}{\eta_{-g}(\tau)}.
\end{aligned}
\end{equation}

Similarly, we combine \cref{zga,zaf,zwf} and take the limit as $\gamma \to -1$, to get:
 \begin{equation}\begin{aligned}\text{tr}\Big(\widehat{g}\mathfrak{z}q^{L_1(0)-\frac{c_1}{24}}y^{J_{11}(0)}\gamma^{J_{12}(0)}&\mid B\otimes A(\mathfrak{a})\Big)\\
 &=q^{-\frac{1}{2}}\prod\limits_{n>0} (1-y^{-1}q^{n-\frac{1}{2}})^2 (1-yq^{n-\frac{1}{2}})^2 \prod\limits_{i=1}^{10} (1-\lambda_i^{-1}q^{n-\frac{1}{2}})(1-\lambda_iq^{n-\frac{1}{2}})\\
&=\dfrac{\theta_4^2(\tau,z)}{\theta_4^2(\tau,0)}\dfrac{\eta_{g}(\tau/2)}{\eta_{g}(\tau)}.
\end{aligned}
\end{equation}

Now that we have computed the traces of $\widehat{g}\mathfrak{z}y^{J_1(0)}q^{L_1(0)-c_1/24}$ and $\widehat{g}y^{J_1(0)}q^{L_1(0)-c_1/24}$ on $B \otimes A(\mathfrak{a})$, we can compute the projection onto $(B \otimes A(\mathfrak{a}))^1$ as follows:
\begin{equation}
 \begin{aligned} \text{tr}\Big(\widehat{g}\mathfrak{z}q^{L_1(0)-\frac{c_1}{24}}y^{J_{11}(0)}\gamma^{J_{12}(0)}&\mid (B\otimes A(\mathfrak{a}))^1\Big) \\ \nonumber &=\dfrac{1}{2}\Big(\text{tr}\left(\widehat{g}\mathfrak{z}q^{L_1(0)-c_1/24}y^{J_{11}(0)}\gamma^{J_{12}(0)}\mid B\otimes A(\mathfrak{a})\right) \\
 &-\text{tr}\left(\widehat{g}q^{L_1(0)-c_1/24}y^{J_{11}(0)}\gamma^{J_{12}(0)}\mid B\otimes A(\mathfrak{a})\right) \Big).
\end{aligned}
\end{equation}
Thus, letting $\gamma \to -1$, we have:
\begin{equation}
 \text{tr}\Big(\widehat{g}\mathfrak{z}q^{L_1(0)-\frac{c_1}{24}}y^{J_{11}(0)}\gamma^{J_{12}(0)}\mid (B\otimes A(\mathfrak{a}))^1\Big)\to \dfrac{1}{2}\left(\dfrac{\theta_4^2(\tau,z)}{\theta_4^2(\tau,0)}\dfrac{\eta_{g}(\tau/2)}{\eta_{g}(\tau)}-\dfrac{\theta_3^2(\tau,z)}{\theta_3^2(\tau,0)}\dfrac{\eta_{-g}(\tau/2)}{\eta_{-g}(\tau)}\right).
 \end{equation}\end{proof}

\begin{Lemma} Let $g\in M_{24}$ such that $\widehat{g}$ fixes a $4$-dimensional space of $\mathfrak{a}$. Let $\mathfrak{z}$ denote the parity involution, let $c_1$ be the central charge, let $L_1(0)$ and $J_{11}(0)$ and $J_{12}(0)$ be operators as before, then 
\begin{equation}\label{qW1}  -\lim\limits_{\gamma \to -1} \emph{tr}\left(\widehat{g}\mathfrak{z}y^{J_{11}(0)}\gamma^{J_{12}(0)}q^{L_1(0)-\frac{c_1}{24}}\mid  (B \otimes A(\mathfrak{a}))^0_{\emph{tw}}\right)=-\dfrac{1}{2}\left(D_g \eta_g(\tau)\dfrac{\theta_1^2(\tau,z)}{\eta^6(\tau)}+C_{-g}\eta_{-g}(\tau)\dfrac{\theta^2_2(\tau,z)}{\theta^2_2(\tau,0)}\right).
\end{equation}
\end{Lemma}

\begin{proof} We begin by recalling the projection operator $P^0(g)=\frac{1}{2}(g+\mathfrak{z}g)$. This will allow us to compute the graded trace \eqref{qW1} on $(B \otimes A(\mathfrak{a}))_{\text{tw}}^0$ by using the traces of $\widehat{g}\mathfrak{z}y^{J_1(0)}q^{L_1(0)-\frac{c_1}{24}}$ and $\widehat{g}y^{J_1(0)}q^{L_1(0)-\frac{c_1}{24}}$ on $(B \otimes A(\mathfrak{a}))_{\text{tw}}$.

We will first compute the traces of $\widehat{g}q^{L_1(0)-c/24}$ and $\widehat{g}\mathfrak{z}q^{L_1(0)-c/24}$ on $A(\mathfrak{a})_{\text{tw}}$. The traces are as follows:
\begin{equation}\label{gatw}\text{tr}\left(\widehat{g}q^{L_{13}(0)-\frac{c_{13}}{24}}\mid A(\mathfrak{a})_{\text{tw}}\right)= \nu q\prod\limits_{n>0}(1+q^{n-1})^2(1+q^{n})^2\prod\limits_{i=1}^{10} (1+\lambda_iq^{n})(1+\lambda_i^{-1}q^{n-1}),
\end{equation}

\begin{equation}\label{zgatw}\text{tr}\left(\widehat{g}\mathfrak{z}q^{L_{13}(0)-\frac{c_{13}}{24}}\mid A(\mathfrak{a})_{\text{tw}}\right)= \nu q\prod\limits_{n>0}(1-q^{n-1})^2(1-q^{n})^2\prod\limits_{i=1}^{10} (1-\lambda_iq^{n})(1-\lambda_i^{-1}q^{n-1}).
\end{equation}

Note that $\widehat{g}$ acts trivially on the components of $B_{\text{tw}}$ because $\mathfrak{f}$ is fixed by $\widehat{g}$. So we compute the traces on the components of $B_{\text{tw}}$ as follows: 
\begin{equation}\label{aftw} \text{tr}\left(y^{J_{11}(0)}q^{L_{11}(0)-\frac{c_{11}}{24}} \mid A(\mathfrak{f})_{\text{tw}}\right)=yq^{\frac{1}{6}}\prod\limits_{n>0} (1+y^{-1}q^{n-1})^{2}(1+yq^{n})^2,
\end{equation}

\begin{equation}\label{zaftw} \text{tr}\left(\mathfrak{z}y^{J_{11}(0)}q^{L_{11}(0)-\frac{c_{11}}{24}} \mid A(\mathfrak{f})_{\text{tw}}\right)=yq^{\frac{1}{6}}\prod\limits_{n>0} (1-y^{-1}q^{n-1})^{2}(1-yq^{n})^2,
\end{equation}
\begin{equation}\label{wftw} \text{tr}\left(\gamma^{J_{12}(0)}q^{L_{12}(0)-\frac{c_{12}}{24}} \mid \mathdutchcal{W}(\mathfrak{f})_{\text{tw}}\right)=\gamma^{-1}q^{-\frac{1}{6}}\prod\limits_{n>0} (1-\gamma^{-1}q^{n-1})^{-2}(1-\gamma q^{n})^{-2},
\end{equation}
and
\begin{equation}\label{zwftw} \text{tr}\left(\mathfrak{z}\gamma^{J_{12}(0)}q^{L_{12}(0)-\frac{c_{12}}{24}} \mid \mathdutchcal{W}(\mathfrak{f})_{\text{tw}}\right)=\gamma^{-1}q^{-\frac{1}{6}}\prod\limits_{n>0} (1+\gamma^{-1}q^{n-1})^{-2}(1+\gamma q^{n})^{-2}.
\end{equation}
As before, we have the operators $L_1(0):=L_{11}(0)+L_{12}(0)+L_{13}(0)$ and the central charge $c_1:=c_{11}+c_{12}+c_{13}$.

We combine \cref{gatw,aftw,wftw} and let $\gamma \to -1$ to compute
\begin{equation}\begin{aligned}\text{tr}\Big(\widehat{g}q^{L_1(0)-\frac{c_1}{24}}y^{J_{11}(0)}\gamma^{J_{12}(0)}&\mid (B\otimes A(\mathfrak{a}))_{\text{tw}}\Big)\\
&=-y\nu q\prod\limits_{n>0}\prod\limits_{i=1}^{10}(1+\lambda_iq^n)(1+\lambda_iq^{n-1})(1+y^{-1}q^{n-1})^2(1+yq^n)^2\\
 &=-C_{-g}\eta_{-g}(\tau)\dfrac{\theta_2^2(\tau,z)}{\theta_2^2(\tau,0)},
\end{aligned}
\end{equation}

and similarly, we combine \cref{zgatw,zaftw,zwftw} and take the limit as $\gamma \to -1$ to compute
\begin{equation}\label{converge}\begin{aligned}\text{tr}\Big(\widehat{g}\mathfrak{z}q^{L_1(0)-\frac{c_1}{24}}y^{J_{11}(0)}\gamma^{J_{12}(0)}&\mid (B\otimes A(\mathfrak{a}))_{\text{tw}}\Big)\\
 &=-y\nu q\prod\limits_{n>0}\prod\limits_{i=1}^{10}(1-\lambda_iq^n)(1-\lambda_iq^{n-1})(1-y^{-1}q^{n-1})^2(1-yq^n)^2\\
 &=D_g \eta_g(\tau)\dfrac{\theta_1^2(\tau,z)}{\eta^6(\tau)}.
\end{aligned}
\end{equation}
We note that letting $\gamma \to -1$ in equation \eqref{converge} does not cause any problems with convergence because the double pole that results from taking this limit in the $n=1$ term of $(1-q^{n-1})^{-2}$ from \cref{zwftw} is canceled by the double zero coming from the $n=1$ term $(1-q^{n-1})^2$ in \cref{zgatw}.

Now that we have computed the traces of $\widehat{g}\mathfrak{z}y^{J_1(0)}q^{L_1(0)-\frac{c_1}{24}}$ and $\widehat{g}y^{J_1(0)}q^{L_1(0)-\frac{c_1}{24}}$ on $(B \otimes A(\mathfrak{a}))_{\text{tw}}$, we can compute the projection onto $(B \otimes A(\mathfrak{a}))_{\text{tw}}^0$ as follows:

  \begin{align} \text{tr}\Big(\widehat{g}\mathfrak{z}q^{L_1(0)-\frac{c_1}{24}}y^{J_{11}(0)}\gamma^{J_{12}(0)}&\mid (B\otimes A(\mathfrak{a}))^0_{\text{tw}}\Big) \\ \nonumber &=\dfrac{1}{2}\Big(\text{tr}\left(\widehat{g}\mathfrak{z}q^{L_1(0)-\frac{c_1}{24}}y^{{J_{11}(0)}\gamma^{J_{12}(0)}}\mid (B\otimes A(\mathfrak{a}))_{\text{tw}}\right)\\ \nonumber&-\text{tr}\left(\widehat{g}q^{L_1(0)-\frac{c_1}{24}}y^{J_{11}(0)}\gamma^{J_{12}(0)}\mid (B\otimes A(\mathfrak{a}))_{\text{tw}}\right) \Big).
\end{align}
Thus, letting $\gamma \to -1$, we get
 \begin{align} \text{tr}\Big(\widehat{g}\mathfrak{z}q^{L_1(0)-\frac{c_1}{24}}y^{J_{11}(0)}\gamma^{J_{12}(0)}&\mid (B\otimes A(\mathfrak{a}))^0_{\text{tw}}\Big)\to\dfrac{1}{2}\left(D_g \eta_g(\tau)\dfrac{\theta_1^2(\tau,z)}{\eta^6(\tau)}+C_{-g}\eta_{-g}(\tau)\dfrac{\theta_2^2(\tau,z)}{\theta_2^2(\tau,0)}\right).
 \end{align}\end{proof}
 
 \begin{Lemma}\label{phi4} Let $\widehat{g}$, $\mathfrak{z}$, $J_{11}(0)$, $J_{12}(0)$, $L_1(0)$, let $c_1$ be as before, then we have \begin{align} -\lim\limits_{\gamma \to -1}\emph{tr}\Big(\widehat{g}\mathfrak{z}q^{L_1(0)-\frac{c_1}{24}}y^{J_{11}(0)}\gamma^{J_{12}(0)}&\mid T\Big)\nonumber \\  &=-\dfrac{1}{2}\left(\dfrac{\theta_4^2(\tau,z)}{\theta_4^2(\tau,0)}\dfrac{\eta_{g}(\tau/2)}{\eta_{g}(\tau)}-\dfrac{\theta_3^2(\tau,z)}{\theta_3^2(\tau,0)}\dfrac{\eta_{-g}(\tau/2)}{\eta_{-g}(\tau)}\right)\\
&-\dfrac{1}{2}\left(D_g \eta_g(\tau)\dfrac{\theta_1^2(\tau,z)}{\eta^6(\tau)}+C_{-g}\eta_{-g}(\tau)\dfrac{\theta_2^2(\tau,z)}{\theta_2^2(\tau,0)}\right) \nonumber \\
&=\phi_g(\tau,z).\nonumber \end{align}
 
 \end{Lemma}
We omit the proof of this lemma because the statement follows from the two lemmas immediately before it.

Now we have shown that the module $T$ recovers the trace functions $\phi_g(\tau,z)$.

Define the operators $L(0):=L_1(0)+L_2(0)+L_3(0)$ and $J(0):=J_{11}(0)+J_3(0)$ and the central charge $c:=c_1+c_2+c_3$. Combining Lemmas \ref{clifflem}, \ref{W3}, and \ref{phi4} we can state the following theorem.
\begin{Theorem} $\widetilde{A}(\mathfrak{p})^{}_{\emph{tw}}\otimes \mathdutchcal{W}(\mathfrak{b})_{\emph{tw}} \otimes T$ is an infinite dimensional, bigraded, virtual module with trace functions as follows:
\begin{equation}
\lim\limits_{\gamma \to -1} \emph{tr}\left(\widehat{g}\mathfrak{z}\widetilde{p}(0)\gamma^{J_{12}(0)}y^{J(0)}q^{L(0)-\frac{c}{24}}\mid \widetilde{A}(\mathfrak{p})^{}_{\emph{tw}}\otimes \mathdutchcal{W}(\mathfrak{b})_{\emph{tw}} \otimes T\right)=\widetilde{M}_g(\tau,z).\end{equation}
\end{Theorem}

\begin{Example}
$\widetilde{A}(\mathfrak{p})^{}_{\text{tw}}\otimes \mathdutchcal{W}(\mathfrak{b})_{\text{tw}} \otimes T$ is a (virtual) module for the group $\bf{M_{22}\colon 2}$, one of the maximal subgroups of $M_{24}$. One can see via the following fusion of conjugacy classes: $$\{ 1A, 2A, 3A, 4B, 4B, 5A, 6A, 7A, 7B, 8A,11A, 2A, 2B, 4A, 4B, 6A, 8A, 10A, 12A, 14A, 14B \}$$ and from looking at the cycle structure of each of these conjugacy classes that each element of $\bf{M_{22}\colon 2}$ fixes a $4$-dimensional space. (Although all of $\bf{M_{22}\colon 2}$ only fixes a $2$-dimensional space).
\end{Example}

 \begin{Example}
 $\widetilde{A}(\mathfrak{p})^{}_{\text{tw}}\otimes \mathdutchcal{W}(\mathfrak{b})_{\text{tw}} \otimes V^{s\natural}_{\text{tw}}$ is a (virtual) module for the group $\bf{2^4\colon A_7}$, one of the maximal subgroups of $M_{23}$. One can see this via the following fusion of conjugacy classes:
  \[\{ 1A, 2A, 2A, 4B, 3A, 3A, 6A, 4B, 8A, 5A, 6A, 7A, 14A, 7B, 14B \}\]
 and from looking at the cycle structure of each of these conjugacy classes that each element of $\bf{2^4\colon A_7}$ fixes a $4$-dimensional space. (Although all of $\bf{2^4\colon A_7}$ only fixes a $3$-dimensional space).
 \end{Example}

\begin{Example}
$\widetilde{A}(\mathfrak{p})^{}_{\text{tw}}\otimes \mathdutchcal{W}(\mathfrak{b})_{\text{tw}} \otimes T$ is a (virtual) module for the group $\bf{A_8}$, one of the maximal subgroups of $M_{23}$. One can see via the following fusion of conjugacy classes: $$\{ 1A, 2A, 2A, 3A, 3A, 4B, 4B, 5A, 6A, 6A, 7A, 7B,15A, 15B \}$$ and from looking at the cycle structure of each of these conjugacy classes that each element of $\bf{A_8}$ fixes a $4$-dimensional space. (Although all of $\bf{A_8}$ only fixes a $3$-dimensional space).
 \end{Example}

\begin{Example}
$\widetilde{A}(\mathfrak{p})^{}_{\text{tw}}\otimes \mathdutchcal{W}(\mathfrak{b})_{\text{tw}} \otimes T$ is a (virtual) module for the smallest sporadic group $\bf{M_{11}}$, one of the subgroups of $M_{24}$. One can see via the following fusion of conjugacy classes: $$\{ 1A, 2A, 3A, 4B, 5A, 6A, 8A, 8A, 11A, 11A \}$$ and from looking at the cycle structure of each of these conjugacy classes that each element of $\bf{M_{11}}$ fixes a $4$-dimensional space. (Although all of $\bf{M_{11}}$ only fixes a $3$-dimensional space).
 \end{Example}
 
 \begin{remark} The module for $\bf{M_{11}}$ restricts in particular to a module for $\bf{\mathbb{Z}/ 11 \mathbb{Z}}$. This gives an explicit realization of the module in \cite{QM} for which the integrality of its trace functions is equivalent to divisibility conditions on the number of $\mathbb{F}_p$ points on $J_0(11)$ because of the cusp forms in the expressions (see equation 3.2 and Appendix A of \cite{QM}).
 
 Similarly the modules for $\bf{M_{22}\colon 2}$ and $\bf{2^4\colon A_7}$ give explicit realizations of modules for which the integrality of their trace functions are equivalent to divisibility conditions on the number of $\mathbb{F}_p$ points on $J_0(14)$ and the module for $\bf{A_8}$ gives an explicit realization of a module for which the integrality of its trace functions is equivalent to divisibility conditions on the number of $\mathbb{F}_p$ points on $J_0(15)$.
 \end{remark}

 \bibliography{MMMbibtex}{}

\begin{thebibliography}{10}

\bibitem{ACH}
Vassilis Anagiannis, Miranda C.~N. Cheng, and Sarah~M. Harrison.
\newblock {$K3$} elliptic genus and an umbral moonshine module.
\newblock {\em Comm. Math. Phys.}, 366(2):647--680, 2019.

\bibitem{QM}
Lea Beneish.
\newblock Quasimodular moonsine and arithmetic connections.
\newblock {\em Transactions of the American Mathematical Society},
  372(12):8793--8813, December 2019.

\bibitem{BKM}
Richard~E. Borcherds.
\newblock Generalized {K}ac-{M}oody algebras.
\newblock {\em J. Algebra}, 115(2):501--512, 1988.

\bibitem{Borcherds}
Richard~E. Borcherds.
\newblock Monstrous moonshine and monstrous {L}ie superalgebras.
\newblock {\em Invent. Math.}, 109(2):405--444, 1992.

\bibitem{hmfbook}
Kathrin Bringmann, Amanda Folsom, Ken Ono, and Larry Rolen.
\newblock {\em Harmonic {M}aass forms and mock modular forms: theory and
  applications}, volume~64 of {\em American Mathematical Society Colloquium
  Publications}.
\newblock American Mathematical Society, Providence, RI, 2017.

\bibitem{2}
Miranda C.~N. Cheng.
\newblock {$K3$} surfaces, {$N=4$} dyons and the {M}athieu group {$M_{24}$}.
\newblock {\em Commun. Number Theory Phys.}, 4(4):623--657, 2010.

\bibitem{holographic}
Miranda C.~N. Cheng and John F.~R. Duncan.
\newblock On {R}ademacher sums, the largest {M}athieu group and the holographic
  modularity of moonshine.
\newblock {\em Commun. Number Theory Phys.}, 6(3):697--758, 2012.

\bibitem{CDMeroJacobi}
Miranda C.~N. Cheng and John F.~R. Duncan.
\newblock Meromorphic {J}acobi forms of half-integral index and umbral
  moonshine modules.
\newblock 07 2017.

\bibitem{umbralconjecture}
Miranda C.~N. Cheng, John F.~R. Duncan, and Jeffrey~A. Harvey.
\newblock Umbral moonshine.
\newblock {\em Commun. Number Theory Phys.}, 8(2):101--242, 2014.

\bibitem{CDHLattices}
Miranda C.~N. Cheng, John F.~R. Duncan, and Jeffrey~A. Harvey.
\newblock Umbral moonshine and the {N}iemeier lattices.
\newblock {\em Res. Math. Sci.}, 1:Art. 3, 81, 2014.

\bibitem{ConwayonLeech}
John~H. Conway.
\newblock A characterisation of {L}eech's lattice.
\newblock {\em Invent. Math.}, 7:137--142, 1969.

\bibitem{conway1985atlas}
John~H. Conway.
\newblock {\em Atlas of finite groups: maximal subgroups and ordinary
  characters for simple groups}.
\newblock Oxford University Press, 1985.

\bibitem{ConwayNorton}
John~H. Conway and Simon~P. Norton.
\newblock Monstrous moonshine.
\newblock {\em Bull. London Math. Soc.}, 11(3):308--339, 1979.

\bibitem{DMZ}
Atish Dabholkar, Sameer Murthy, and Don Zagier.
\newblock Quantum black holes, wall crossing, and mock modular forms.
\newblock {\em arXiv:1208.4074, to appear in Cambridge Monographs in
  Mathematical Physics}, 08 2012.

\bibitem{JohnDuke}
John F.~R. Duncan.
\newblock Super-moonshine for {C}onway's largest sporadic group.
\newblock {\em Duke Math. J.}, 139(2):255--315, 2007.

\bibitem{DGO}
John F.~R. Duncan, Michael~J. Griffin, and Ken Ono.
\newblock Proof of the umbral moonshine conjecture.
\newblock {\em Research in the Mathematical Sciences}, 2:26, 2015.

\bibitem{DH}
John F.~R. Duncan and Jeffrey~A. Harvey.
\newblock The umbral moonshine module for the unique unimodular {N}iemeier root
  system.
\newblock {\em Algebra Number Theory}, 11(3):505--535, 2017.

\bibitem{DMCconway}
John F.~R. Duncan and Sander Mack-Crane.
\newblock The moonshine module for {C}onway's group.
\newblock {\em Forum Math. Sigma}, 3:e10, 52, 2015.

\bibitem{DMCderived}
John F.~R. Duncan and Sander Mack-Crane.
\newblock Derived equivalences of {K}3 surfaces and twined elliptic genera.
\newblock {\em Res. Math. Sci.}, 3:Paper No. 1, 47, 2016.

\bibitem{DO}
John F.~R. Duncan and Andrew O'Desky.
\newblock Super vertex algebras, meromorphic {J}acobi forms and umbral
  moonshine.
\newblock {\em Journal of Algebra}, 515:389 -- 407, 2018.

\bibitem{5}
Tohru Eguchi and Kazuhiro Hikami.
\newblock Note on twisted elliptic genus of {$K3$} surface.
\newblock {\em Phys. Lett. B}, 694(4-5):446--455, 2011.

\bibitem{EOT}
Tohru Eguchi, Hirosi Ooguri, and Yuji Tachikawa.
\newblock Notes on the {$K3$} surface and the {M}athieu group {$M_{24}$}.
\newblock {\em Exp. Math.}, 20(1):91--96, 2011.

\bibitem{folsom}
Amanda Folsom.
\newblock Perspectives on mock modular forms.
\newblock {\em J. Number Theory}, 176:500--540, 2017.

\bibitem{FrenkelBenzvi}
Edward Frenkel and David Ben-Zvi.
\newblock {\em Vertex algebras and algebraic curves}, volume~88 of {\em
  Mathematical Surveys and Monographs}.
\newblock American Mathematical Society, Providence, RI, second edition, 2004.

\bibitem{FLH}
Igor~B. Frenkel, Yi-Zhi Huang, and James Lepowsky.
\newblock On axiomatic approaches to vertex operator algebras and modules.
\newblock {\em Mem. Amer. Math. Soc.}, 104(494):viii+64, 1993.

\bibitem{FLM84}
Igor~B. Frenkel, James Lepowsky, and Arne Meurman.
\newblock A natural representation of the {F}ischer-{G}riess {M}onster with the
  modular function {$J$} as character.
\newblock {\em Proc. Nat. Acad. Sci. U.S.A.}, 81(10, Phys. Sci.):3256--3260,
  1984.

\bibitem{FLM83}
Igor~B. Frenkel, James Lepowsky, and Arne Meurman.
\newblock A moonshine module for the {M}onster.
\newblock In {\em Vertex operators in mathematics and physics ({B}erkeley,
  {C}alif., 1983)}, volume~3 of {\em Math. Sci. Res. Inst. Publ.}, pages
  231--273. Springer, New York, 1985.

\bibitem{FLM}
Igor~B. Frenkel, James Lepowsky, and Arne Meurman.
\newblock {\em Vertex operator algebras and the {M}onster}, volume 134 of {\em
  Pure and Applied Mathematics}.
\newblock Academic Press, Inc., Boston, MA, 1988.

\bibitem{4}
Matthias~R. Gaberdiel, Stefan Hohenegger, and Roberto Volpato.
\newblock Mathieu {M}oonshine in the elliptic genus of {$K3$}.
\newblock {\em J. High Energy Phys.}, (10):062, 24, 2010.

\bibitem{3}
Matthias~R. Gaberdiel, Stefan Hohenegger, and Roberto Volpato.
\newblock Mathieu twining characters for {$K3$}.
\newblock {\em J. High Energy Phys.}, (9):058, 20, 2010.

\bibitem{Gannon}
Terry Gannon.
\newblock Much ado about {M}athieu.
\newblock {\em Adv. Math.}, 301:322--358, 2016.

\bibitem{Griess}
Robert~L. Griess, Jr.
\newblock The friendly giant.
\newblock {\em Invent. Math.}, 69(1):1--102, 1982.

\bibitem{Leech1}
John Leech.
\newblock Some sphere packings in higher space.
\newblock {\em Canadian J. Math.}, 16:657--682, 1964.

\bibitem{Leech2}
John Leech.
\newblock Notes on sphere packings.
\newblock {\em Canadian J. Math.}, 19:251--267, 1967.

\bibitem{Niemeier}
Hans-Volker Niemeier.
\newblock Definite quadratische {F}ormen der {D}imension {$24$} und
  {D}iskriminante {$1$}.
\newblock {\em J. Number Theory}, 5:142--178, 1973.

\bibitem{TW10}
Anne Taormina and Katrin Wendland.
\newblock The symmetries of the tetrahedral {K}ummer surface in the {M}athieu
  group ${M}_{24}$.
\newblock {\em arXiv preprint arXiv:1008.0954}, 2010.

\bibitem{TW13}
Anne Taormina and Katrin Wendland.
\newblock The overarching finite symmetry group of {K}ummer surfaces in the
  {M}athieu group {$M_{24}$}.
\newblock {\em J. High Energy Phys.}, (8):125, front matter+62, 2013.

\bibitem{TW15a}
Anne Taormina and Katrin Wendland.
\newblock Symmetry-surfing the moduli space of {K}ummer {K}3s.
\newblock In {\em String-{M}ath 2012}, volume~90 of {\em Proc. Sympos. Pure
  Math.}, pages 129--153. Amer. Math. Soc., Providence, RI, 2015.

\bibitem{TW15b}
Anne Taormina and Katrin Wendland.
\newblock A twist in the {$M_{24}$} {M}oonshine story.
\newblock {\em Confluentes Math.}, 7(1):83--113, 2015.

\bibitem{TaoWend}
Anne Taormina and Katrin Wendland.
\newblock {SU(2)} channels the cancellation of {K3} {BPS} states.
\newblock {\em arXiv preprint arXiv:1908.03148}, 2019.

\bibitem{thompsonfmf}
John~G. Thompson.
\newblock Finite groups and modular functions.
\newblock {\em Bull. London Math. Soc.}, 11(3):347--351, 1979.

\bibitem{Zweg}
Sander~P. Zwegers.
\newblock {\em Mock Theta Functions}.
\newblock PhD thesis, Universiteit Utrecht, 2002.

\end{thebibliography}
\bibliographystyle{plain}

\end{document}